\documentclass[12pt, letterpaper]{amsart}
\usepackage{amsfonts,amsthm,mathtools}
\usepackage[all,cmtip]{xy}
\usepackage{fullpage}
\usepackage{amsmath}
\usepackage{amssymb}
\usepackage{enumerate}
\usepackage{tikz}
\usepackage[backref]{hyperref}
\usepackage{cleveref}
\usepackage{verbatim}
\usepackage{cancel}
\usepackage{float}

\newtheorem{lem}{Lemma}[section]
\newtheorem{thm}[lem]{Theorem}

\newtheorem{prop}[lem]{Proposition}
\newtheorem{cor}[lem]{Corollary}
\newtheorem{conj}[lem]{Conjecture}

\theoremstyle{definition}
\newtheorem{remark}[lem]{Remark}

\DeclareMathAlphabet{\curly}{U}{rsfs}{m}{n}

\newcommand{\tors}{\operatorname{tors}}

\newcommand{\Aut}{\operatorname{Aut}}
\newcommand{\Gal}{\operatorname{Gal}}

\newcommand{\Pic}{\operatorname{Pic}}
\newcommand{\gon}{\operatorname{gon}}

\newcommand{\Q}{\mathbb{Q}}
\newcommand{\C}{\mathbb{C}}

\newcommand{\Z}{\mathbb{Z}}
\newcommand{\F}{\mathbb{F}}

\newcommand{\SL}{\operatorname{SL}}

\newcommand{\Div}{\operatorname{Div}}
\newcommand{\red}{\operatorname{red}}

\def\diam#1{\langle#1\rangle}
\newcommand{\fp}{\mathfrak p}

\DeclareMathOperator{\dv}{{div}}

\newcommand{\PP}{{\mathbb P}}
\DeclareMathOperator{\cf}{{Cliff}}
\DeclareMathOperator{\cd}{{CD}}

\mathchardef\mhyphen="2D

\newtheorem*{ack}{Acknowledgements}

\title{Gonality of the modular curve $X_0(N)$}

\author{\sc Filip Najman}
\address{Filip Najman \\
University of Zagreb\\  
Bijeni\v{c}ka Cesta 30 \\
10000 Zagreb\\
Croatia}
\email{fnajman@math.hr}
\urladdr{http://web.math.pmf.unizg.hr/~fnajman}

\author{\sc Petar Orli\'c}
\address{Petar Orli\'c \\
University of Zagreb\\  
Bijeni\v{c}ka Cesta 30 \\
10000 Zagreb\\
Croatia}
\email{petar.orlic@math.hr}

\begin{document}
\begin{abstract}
    In this paper we determine the $\Q$-gonalities of the modular curves $X_0(N)$ for all $N<145$. We determine the $\C$-gonality of many of these curves and the $\Q$-gonalities and $\C$-gonalities for many larger values of $N$. 
    
    Using these results and some further work, we determine all the modular curves $X_0(N)$ of gonality $4$, $5$ and $6$ over $\Q$. We also find the first known instances of pentagonal curves $X_0(N)$ over $\C$.
\end{abstract}


\thanks{The authors were supported by the QuantiXLie Centre of Excellence, a
  project co-financed by the Croatian Government and European Union
  through the European Regional Development Fund - the Competitiveness
  and Cohesion Operational Programme (Grant KK.01.1.1.01.0004) and by
  the Croatian Science Foundation under the project
  no. IP-2018-01-1313.}

\maketitle

\section{Introduction}
Let $k$ be a field and $C$ a curve over $k$ (throughout the paper we assume all curves are geometrically integral). The \textit{gonality} $\gon_k C$ of $C$ over $k$ is defined to be the least degree of a non-constant morphism $f:C\rightarrow \PP^1$, or equivalently the least degree of a non-constant function $f\in k(C)$.

Gonalities of modular curves and their quotients have been the subject of extensive research by many people. The study of gonalities of the classical modular curve $X_0(N)$ started with Ogg \cite{Ogg74}, who determined the hyperelliptic modular curves $X_0(N)$. Hasegawa and Shimura \cite{HasegawaShimura_trig} determined both the $X_0(N)$ that are trigonal over $\C$ and the $X_0(N)$ that are trigonal over $\Q$ and Jeon and Park \cite{JeonPark05} determined the $X_0(N)$ that are tetragonal over $\C$. More generally, Abramovich \cite{abramovich} gave a lower bound for the gonality over $\C$ for any modular curve (which is usually not sharp). 

In this paper we will study the gonalities of the modular curves $X_0(N)$ over $\Q$ instead of over $\C$. The motivation for this comes from two directions. 

Firstly, the $\Q$-gonality of a curve is arguably more interesting from an arithmetical point of view than its $\C$-gonality. For example, when one wants to determine the modular curves $X_0(N)$ and $X_1(M,N)$ with infinitely many degree $d$ points (over $\Q$), a question of fundamental arithmetical importance, as these curves parametrise elliptic curves with level structures, then determining all such curves of gonality $d$ plays a key role. 

Using the gonality of the modular curves as one of the main ingredients, all the modular curves $X_1(M,N)$ with infinitely many degree $d$ points have been determined for $d=2$ by Mestre \cite{Mestre81}, $d=3$ by Jeon, Kim and Schweizer \cite{JeonKimSchweizer04}, for $d=4$ by Jeon, Kim and Park \cite{JeonKimPark06} and for $d=5,6$ by Derickx and Sutherland \cite{DerickxSutherland17}. The same problem has been solved for the modular curves $X_0(N)$ for $d=2$ by Bars \cite{Bars99} and for $d=3$ by Jeon \cite{Jeon2021}. 


The other motivation comes from the database, which is in construction, of modular curves that will be incorporated in the LMFDB \cite{lmfdb}, which tabulates $L$-functions, modular forms, elliptic curves and related objects. At the moment of writing of this paper, there were 115 modular curves $X_0(N)$ in LMFDB, all with $N\leq 127$. The exact $\Q$-gonality was listed as known for less than half of them. Our work determines the $\Q$-gonality for all the $N$ in this database. 

Although our interest lies primarily in $\Q$-gonalities, we compute and document the $\C$-gonality wherever possible. Our main result is the following theorem.

\begin{thm}\label{main_tm}
The $\Q$-gonalities and $\C$-gonalities of $X_0(N)$ are as listed in \Cref{tab:main}, \Cref{tab:main2} and \Cref{tab:large}.
\end{thm}

One immediate consequence of our result is the determination of all $X_0(N)$ that are tetragonal over $\Q$. A curve that is tetragonal over $\Q$ has to have gonality $\leq 4$ over $\C$ and all curves satisfying this are known by the aforementioned results \cite{Ogg74, HasegawaShimura_trig, JeonPark05}. As we determine in \Cref{main_tm} all $N$ satisfying this and that have gonality $4$ over $\Q$ the following result immediately follows.

\begin{thm}\label{thm:tetragonal}
The modular curve $X_0(N)$ is tetragonal over $\Q$ if and only if 
\begin{align*}N\in \{&38,42,44,51,52,53,55,56,57,58,60,61,62,63,65,66,67,68,69,70,72,73,74,75,\\
&
77,78,79,80,83,85,87,88,89,91,92,94,95,96,98,100,101,103,104,107,111,119,\\
&121,125,131,142,143,167,191\}.\end{align*}
\end{thm}

After \Cref{thm:tetragonal} and the aforementioned results \cite{Ogg74, HasegawaShimura_trig, JeonPark05} which determine all $X_0(N)$ of $\C$-gonality $\leq 4$, the question of determining the pentagonal, and after that, hexagonal (both over $\Q$ and over $\C$) curves $X_0(N)$ naturally arises. Surprisingly, there seems to have been no known curve that is pentagonal either over $\Q$ or over $\C$ (at least to our knowledge); see \cite[p.139-140]{HasegawaShimura_trig} for a short discussion stating this.  As a byproduct of our results and with some additional work, we determine all $X_0(N)$ that are pentagonal of hexagonal over $\Q$.

\begin{thm}\label{thm:pentagonal}
The modular curve $X_0(N)$ is pentagonal over $\Q$ if and only if $N=109$.
\end{thm}

\begin{thm}\label{thm:hexagonal}
The modular curve $X_0(N)$ is hexagonal over $\Q$ if and only if
\begin{align*} N\in 
\{&76,82,84,86,90,93,97,99,108,112,113,115-118,122-124,127-129,135,\\
&137,139,141,144,146,147,149,151,155,159,162,164,169,179,181,215,227,239\}.\end{align*}
\end{thm}

We also show that $X_0(N)$ is pentagonal over $\C$ for $N=97$ and $169$, obtaining the first known such curves.

Our methods are perhaps most similar to the work of Derickx and van Hoeij \cite{derickxVH}, where they compute the exact $\Q$-gonalities of the modular curves $X_1(N)$ for $N\leq 40$ and give upper bounds for $N\leq 250$. Some of our methods will be similar to the ones in \cite{derickxVH}, but some will be different; the differences arise out of the intrinsic properties of the different modular curves. In particular, on one hand the properties that make $X_0(N)$ easier to deal with is its lower genus and an abundance of involutions (especially for highly composite $N$). On the other hand, $X_0(N)$ has much fewer cusps and hence much fewer modular units, the main tool in \cite{derickxVH} for obtaining upper bounds. Another difficulty with $X_0(N)$, as opposed to $X_1(N)$, is that it is in general hard to obtain reasonable plane models, making computations in function fields much more computationally demanding. 

We now give a brief description of our methods, and compare them to the methods of \cite{derickxVH}. The way to determine the exact gonality of a modular curve is to give a lower bound and an upper bound for the gonality which match each other. We explain both in more detail and rigor in \Cref{sec:lb} and \Cref{sec:ub}.

The authors of \cite{derickxVH} were (perhaps surprisingly) able to use a single method to obtain lower bounds and a single method to obtain upper bounds. The lower bounds were obtained using the well-known fact that $\gon_{\F_p}(C)\leq \gon_{\Q}(C)$ for a prime of good reduction $p$ of $C$ and by then computing $\gon_{\F_p}(C)$, which is a finite computation. This will be one of our main tools too, together with the Castelnuovo-Severi inequality (see \Cref{tm:CS}) and using $\gon_{\C}(C)\leq \gon_{\Q}(C)$ together with known results about $\C$-gonalities. An especially interesting method, described in \Cref{sec:MW}, is Mordell-Weil sieving on the Brill-Noether varieties $W_d^1(X)$, which we use to show that $X_0(97)$ is of gonality 6 and $X_0(133)$ is of gonality 8 over $\Q$.   To produce lower bounds on the $\C$-gonality, we will use computations on Betti numbers and proven parts of the Green conjecture (see e.g. \Cref{nog51}). There will be a few instances in which we also use other methods.

Derickx and van Hoeij obtained their upper bounds by constructing \textit{modular units} (functions whose polar divisor is supported only on cusps) of a certain degree. In certain instances we will also obtain upper bounds by explicitly constructing functions of degree $d$ by searching in Riemann-Roch spaces of sets of divisors with some fixed support. For us the set of divisors through which we will search through will not be supported only on cusps, but will also include CM points and even non-CM non-cuspidal rational points. We will also construct functions on $X_0(N)$ by finding functions $g$ of degree $k$ on $X_0(N)/w_d$ or $X_0(d)$ for $d|N$ and then pulling them back via the quotient map $f:X_0(N)\rightarrow X_0(N)/w_d$ or $f:X_0(N)\rightarrow X_0(d)$, thus obtaining a map $f^*g$ of degree $k \cdot \deg f$. Apart from this we will also use the Tower theorem (see \Cref{cor:TT}) which allows us to determine the $\Q$-gonality from the $\C$-gonality under certain assumptions.

A lot of our results rely on extensive computation in \texttt{Magma} \cite{magma}. To compute models for $X_0(N)$ and their quotients by Atkin-Lehner involutions, we used the code written by Philippe Michaud-Jacobs as part of an ongoing collaborative project on computing points of low degree on modular curves \cite{quad_pts}. 

It is natural to wonder why we stopped at the point where we did and whether one can determine $\Q$-gonalities of $X_0(N)$ for larger $N$. We discuss this, the complexity of the most computationally demanding parts of our computations, and possible further work briefly in \Cref{sec:future work}. 

The code that verifies all our computations can be found on:
\begin{center}
    \url{https://github.com/orlic1/gonality_X0}.
\end{center}
All of our computations were performed on the \textit{Euler} server at the Department of Mathematics, University of Zagreb with a Intel Xeon W-2133 CPU running at 3.60GHz and with 64 GB of RAM.

\ack{We are grateful to Maarten Derickx for the ideas used in \Cref{sec:MW} that helped resolve the cases $X_0(N)$ for $N=97,133$ and $145$, and to Andreas Schweizer for helpful comments. We are grateful to the anonymous referee for many helpful comments that have greatly improved the exposition.}

\section{Notation}

We now set up notation that will be used throughout the paper. Throughout the paper $p$ will be a prime number and $q$ a power of $p$. For a curve $X$ the genus of $X$ will be denoted by $g(X)$. By $X_0(N)$ we denote the classical modular curve parametrizing pairs $(E, C)$ of generalized elliptic curves together with a cyclic subgroup $C$ of order $N$. 

For any divisor $d$ of $N$ such that $(d,N/d)=1$ the \textit{Atkin-Lehner involution} $w_d$ acts on $(E,C)$ by sending it to $(E/C,E[N]/C)$. The Atkin-Lehner involutions form a subgroup of $\Aut(X_0(N))$ isomorphic to $(\Z/2\Z)^{\omega(N)}$, where $\omega(N)$ is the number of prime divisors of $N$. The curve $X_0(N)$ and all its Atkin-Lehner involutions are defined over $\Q$. They induce involutions on $\Div X_0(N)$ and $\Pic^d X_0(N)$ which we will also denote by $w_N$.

The quotient $X_0(N)/w_N$ is denoted by $X_0^+(N)$ and the quotient of $X_0(N)$ by the whole group of Atkin-Lehner involutions is denoted by $X_0^*(N)$. By $J_0(N)$ we denote the Jacobian of $X_0(N)$ and by $J_0(N)^-:=(1-w_N)J_0(N)$.

\section{Lower bounds}
\label{sec:lb}
In this section we give all the results used to obtain lower bounds for the gonalities of $X_0(N)$. We first mention two obvious lower bounds for a curve $C$ defined over a number field $K$:
\begin{equation}\label{lb:C}
\gon_{\C} C\leq \gon_{K} C,
\end{equation}
and
\begin{equation}
\gon_{\F_\fp} C\leq \gon_{K} C,
\end{equation}
where $\fp$ is a prime ideal of good reduction of $C$ and $\F_\fp$ is the residue field of $\fp$. The determination of $\gon_{\F_\fp} C$ is a finite task, although often a computationally difficult one. More explicitly, $\gon_{\F_\fp} C$ can be determined by checking the Riemann-Roch spaces of all degree $d$ effective $\F_\fp$-rational divisors $D$; the smallest $d$ such that there exists such a divisor $D$ of degree $d$ for which $l(D)\geq 2$ is the $\F_\fp$-gonality of $C$. 

We will be interested only in the case $K=\Q$ and $\fp=p$, a rational prime. The following lemma will be useful in making the computation of $\gon_{\F_p} C$ much quicker. 

\begin{lem}\label{lem:F_p_div_search}
Let $C/\F_p$ be a curve such that $\#C(\F_p)=n$. Suppose that there exists a function $f$ of degree $d$ in $\F_p(C)$. Then:
\begin{itemize}
    \item[a)] There exists a function $g$ of degree $d$ such that its polar divisor is supported on at most $\left \lfloor \frac{n}{p+1} \right \rfloor$ points $P\in C(\F_p)$.
    \item[b)] There exists a function $h$ of degree $d$ such that its polar divisor is supported on at least $\left \lceil \frac{n}{p+1} \right \rceil$ points $P\in C(\F_p)$.
\end{itemize}
\end{lem}
\begin{proof}
We will prove a), as b) is proved analogously. As $f$ maps $C(\F_p)$ into $\PP^1(\F_p)$, it follows by the pigeonhole principle that there exists a $c\in \PP^1(\F_p)$ such that $f^{-1}(c)$ consists of at most $\left \lfloor \frac{n}{p+1} \right \rfloor$ points. If $c=\infty$, then let $g:=f$, otherwise we define $g(x):=\frac{1}{f(x)-c}$. The function $g$ obviously satisfies the claim. 
\end{proof}

\begin{prop}\label{poonenlowerbound}
Let $f:X\rightarrow Y$ be a non-constant morphism of curves over $k$. Then $\gon_k(Y)\leq \gon_k(X)$.
\end{prop}
\begin{proof}
This is \cite[Proposition A.1 (vii)]{Poonen2007} or \cite[Lemma 1.3]{NguyenSaito}.
\end{proof}

A very useful tool for producing a lower bound on the gonality is the Castelnuovo-Severi inequality (see \cite[Theorem 3.11.3]{Stichtenoth09} for a proof).
\begin{prop}[Castelnuovo-Severi inequality] 
\label{tm:CS}
Let $k$ be a perfect field, and let $X,\ Y, \ Z$ be curves over $k$. Let non-constant morphisms $\pi_Y:X\rightarrow Y$ and $\pi_Z:X\rightarrow Z$ over $k$ be given, and let their degrees be $m$ and $n$, respectively. Assume that there is no morphism $X\rightarrow X'$ of degree $>1$ through which both $\pi_Y$ and $\pi_Z$ factor. Then the following inequality hold:
\begin{equation} \label{eq:CS}
g(X)\leq m \cdot g(Y)+n\cdot g(Z) +(m-1)(n-1).
\end{equation}
\end{prop}

We now deduce an easy corollary of \Cref{tm:CS}.

\begin{cor} \label{cor:CS}Let $X$ and $Y$ be curves over $k$ such that there exists a non-constant morphism $\pi:X\rightarrow Y$ of degree $2$. If $g(X)-2g(Y)\geq d-1$ and $\gon_k(Y)>\lfloor(d-1)/2 \rfloor$, then $\gon_k(X)\geq d$.
\end{cor}
\begin{proof}
    Suppose the opposite and let $f:X\rightarrow \PP^1$ be a morphism of degree $\leq d-1$. If both $f$ and $\pi$ factor through a morphism $X\rightarrow X'$, this morphism is of degree $2$ and $X'\simeq Y$; this implies that there is a morphism $f:Y\rightarrow \PP^1$ of degree at most $\lfloor (d-1)/2 \rfloor$. This is not yet a contradiction, since we still need to rule out the possibility of such an $f$ being defined over $\overline {k}$, but not over $k$. As $f\circ \pi$ and $\pi$ are defined over $k$ we have
    $$ f \circ \pi =(f\circ \pi)^\sigma = f^\sigma \circ \pi.$$
    Since $\pi$ is a non-constant morphism it is surjective over $\overline k$, so it follows that $f=f^\sigma$ and hence $f$ is defined over $k$.

    On the other hand if $f$ and $\pi$ do not both factor through a morphism of degree $>1$, the Castelnuovo-Severi inequality \eqref{eq:CS} applied to $\pi$ and $f$ gives a contradiction.
\end{proof}
\begin{lem}\label{lem:fp}
Let $X$ be a curve, $p$ a prime of good reduction for $X$ and $q$ a power of $p$. Suppose $\#X(\F_q)> d(q+1)$ for some $d$. Then $\gon_\Q(X)>d.$
\end{lem}
\begin{proof}
Let  $f\in \F_q(X)$ be a function of degree $\leq d$. Then for any $c\in\PP^1(\F_q)$ we have $\#f^{-1}(c)\leq d$. Since $f$ sends $X(\F_q)$ into $\PP^1(\F_q)$, it follows that $\#X(\F_q) \leq d(q+1)$. 
\end{proof}

\subsection{Lower bounds for gonality over $\C$}
Here we use proven parts of Green's conjecture to obtain a lower bound for $\gon_\C X$, which in turn gives us a lower bound on $\gon_\Q X$ by \eqref{lb:C}. We mostly follow the notation of \cite{JeonPark05}, stated in the language of divisors instead of line bundles. A $g_d^r$ is a subspace $V$ of $L(D)$, for a divisor $D$ on $X$, such that $\deg D=d$ and $\dim V=r+1$. Since removing the base locus of a linear series decreases the degree while preserving $r$, the gonality $\gon_\C X$ is the smallest $d$ such that $X$ has a $g_d^1.$

Let $D$ be a divisor on $X$ and $K$ a canonical divisor on $X$. The \textit{Clifford index of $D$} is the integer
$$\cf (D):= \deg D -2(\ell(D)-1),$$
and the \textit{Clifford index of $X$} is 
$$\cf (X):=\min \{ \cf (D) \mid  \ell(D)\geq 2 \text{ and }\ell (K-D)\geq 2\}.$$ 
The Clifford index gives bounds for the $\C$-gonality of $X$ (see \cite{CoppensMartens91}):
\begin{equation}\label{clifford_index_bounds}
\cf(X)+2 \leq \gon_\C X \leq \cf(X)+3.    
\end{equation}

The \textit{Clifford dimension of $X$} is defined to be
$$\cd(X):=\min \{ \ell (D)-1 \mid \cf(D) =\cf(X)\}.$$

 Let $X$ be a non-hyperelliptic curve. It has a canonical embedding $X\hookrightarrow \PP^{g-1}$. Let $S:=\C[X_0, \ldots X_{g-1}]$, let $I_X$ be the ideal of $X$ and $S_X$ be $S$-module $S_X:=S/I_X$. Let
\begin{equation}
    0\rightarrow F_{g-1}\rightarrow \cdots F_2 \rightarrow F_1 \rightarrow S \rightarrow S_X \rightarrow 0
\end{equation}
be the minimal free resolution of $S_X$, where 
$$F_i=\bigoplus_{j\in \Z}S(-i-j)^{\beta_{i,j}}.$$
The numbers $\beta_{i,j}$ are called the \textit{graded Betti numbers}. Green's conjecture relates graded Betti numbers with the existence of $g_d^r$. We state it as in \cite[p.84]{Schreyer91} (note that the indices of Betti numbers are defined differently there).  
\begin{conj}[Green \cite{Green84}] Let $X$ be a curve of genus $g$. Then 
$\beta_{p,2} \neq 0$ if and only if there exists a divisor $D$ on $X$ of degree $d$ such that a subspace $ g_d^r\text{ of } L(D)$ satisfies $d\leq g-1$, $r=\ell(D)-1\geq 1$ and $d-2r \leq p.$
\end{conj}

The ``if" part of the statement has been proved by Green and Lazarsfeld in the appendix of \cite{Green84}.

\begin{thm}[Green and Lazarsfeld, Appendix to \cite{Green84}] Let $X$ be a curve of genus $g$.
If $\beta_{p,2}=0,$ then there does not exist a divisor $D$ on $X$ of degree $d$ such that a subspace $ g_d^r\text{ of } L(D)$ satisfies $d\leq g-1$, $r=\ell(D)-1\geq 1$ and $d-2r \leq p.$
\end{thm}

For the ease of the reader we state the direct consequence of this theorem that we are going to use.
\begin{cor}\label{nog51}
Let $X$ be a curve of genus $g\geq 6$ with $\beta_{3,2}=0$. Then $X$ has no $g_5^1$. 
\end{cor}

\section{Upper bounds}
\label{sec:ub}

In this section we give all the results used to obtain upper bounds for the gonalities of $X_0(N)$. 
\begin{prop}\label{prop_A.1}
Let $X$ be a curve of genus $g$ over a field $k$.\begin{itemize}
\item[(i)] If $X(k)\neq\emptyset$, then $\gon_k(X)\leq g+1$. If $X(k)\neq\emptyset$ and $g\geq2$, then $\gon_k(X)\leq g$.
\item[(ii)] If $k$ is algebraically closed, then $\gon_k(X)\leq \lfloor \frac{g+3}{2} \rfloor$.
\end{itemize}

\end{prop}

\begin{proof}
This is \cite[Proposition A.1 (iv) and (v)]{Poonen2007}.
\end{proof}

\begin{prop}\label{prop:UB}
Let $f:X\rightarrow Y$ be a non-constant morphism of curves over $k$. Then $\gon_k(X)\leq \deg f\cdot \gon_k(Y)$.
\end{prop}

\begin{proof}
This is trivial; see also \cite[Proposition A.1 (vi)]{Poonen2007}.
\end{proof}

\begin{prop}\label{prop:lower_deg}
Let $p$ be a rational prime. There exists a morphism from $X_0(pN)$ to $X_0(N)$ defined over $\Q$ which is of degree $p+1$ if $p\nmid N$ and of degree $p$ if $p\mid N$.
\end{prop}

\begin{proof}
The map $\pi_p:X_0(pN)\rightarrow X_0(N)$ sends the point corresponding to $(E,C_{pN})$, where $C_{pN}$ is a cyclic subgroup of $E$ of order $pN$, to $(E,pC_{pN})$. Thus the degree of $\pi_p$ is the number of points $(E,X)$ that satisfy $\pi_p((E,X))=(E,C_{N})$ for a given a fixed subgroup $C_N$ of $E$. This is equal to the number of cyclic subgroups $X$ of $(\Z/pN\Z)^2$ which satisfy $pX=C_N$, which is easily seen to be as claimed.

\end{proof}

We now state the Tower theorem and two very useful corollaries.

\begin{thm}[The Tower theorem]
\label{tm:TT}
Let $C$ be a curve defined over a perfect field $k$ and $f:C\rightarrow \PP^1$ be a non-constant morphism over $\overline{k}$ of degree $d$. Then there exists a curve $C'$ defined over $k$ and a non-constant morphism $C\rightarrow C'$ defined over $k$ of degree $d'$ dividing $d$ such that 
$$g(C')\leq \left(\frac{d}{d'}-1\right)^2.$$
\end{thm}
\begin{proof}
This is \cite[Theorem 2.1]{NguyenSaito}. For a published proof, see \cite[Proposition 2.4]{Poonen2007}.
\end{proof}

\begin{cor} \label{cor:TT}
Let $C$ be a curve defined over a perfect field $k$ such that $C(k) \neq \emptyset$ and let $f:C\rightarrow \PP^1$ be a non-constant morphism over $\overline{k}$ of prime degree $d$ such that $g(C)>(d-1)^2$. Then there exists a non-constant morphism $C\rightarrow \PP^1$ of degree $d$ defined over $k$.
\end{cor}
\begin{proof}
From \cite[Corollary 1.7.]{HasegawaShimura_trig} it immediately follows that there exists a curve $C'$ of genus $0$ and a non-constant morphism $C\rightarrow C'$ of degree $d$ defined over $k$. Since $C(k)\neq \emptyset$, it follows $C'(k)\neq \emptyset$. Hence $C'$ is isomorphic to $\PP^1$ over $k$, proving our claim. 
\end{proof}

\begin{cor}\label{Jeon-Park}\begin{itemize}
\item[(i)] Let $C$ be a curve over $\Q$ of genus $\geq 5$ which is trigonal over $\C$ and such that $C(\Q) \neq \emptyset$. Then $C$ is trigonal over $\Q$.
\item[(ii)] Let $C$ be a curve defined over $\Q$ with $\gon_\C(X)=4$ and $g(X)\geq 10$ and such that $C(\Q) \neq \emptyset$. Then $\gon_\Q(X)=4$.
\end{itemize}
\end{cor}
\begin{proof}
Part (i) follows immediately from \Cref{cor:TT} by specializing $C'$ to be $\PP^1$ and $d$ to be $3$.

To prove part (ii) we note that, by \Cref{tm:TT}, $C$ will have a map of degree $d'$ over $\Q$ dividing 4 to a curve of genus $\leq (4/d'-1)^2$, so $d'$ cannot be $1$. If $d'$ is $2$, then $X$ is bielliptic (and is tetragonal over $\Q$). If $d'$ is $4$, then $X$ is tetragonal over $\Q$, as required.
\end{proof}

\section{Results}

In this section we apply the results of \Cref{sec:lb} and \Cref{sec:ub} to the modular curves $X_0(N)$ to obtain upper and lower bounds for their gonality. An overview of the results and the location of the proofs for each curve can be found in the tables at the end of the paper. 

\subsection{Upper bounds obtained by searching in Riemann-Roch spaces}

One way of obtaining an upper bound of $d$ on the gonality over $\Q$ of modular curves is to explicitly construct a function $f$ of degree $d$. This can be done by finding an effective $\Q$-rational divisor $D$ such that $\ell(D)\geq 2.$

\begin{prop}\label{prop4.7}
The $\Q$-gonality of $X_0(N)$ for $N=85$ and $88$ is at most $4$.
\end{prop}
\begin{proof}
To prove the upper bound, we construct a function of degree $4$ by looking at the Riemann-Roch spaces of $\Q$-rational divisors of degree $4$ whose support is in the quadratic points obtained by the pullbacks of rational points on $X_0^*(85)$ and $X_0^+(88)$, respectively. We note that in the case $N=85$ we were unable to obtain such functions by looking at pullbacks from rational points on $X_0(N)/w_d$, for any of the Atkin-Lehner involutions.
\end{proof}

\begin{prop}\label{prop:trig_quot}
The genus $4$ quotients $X_0^+(N)$ are trigonal over $\Q$ for $$N=84,93,115,116,129,137,155,159.$$
\end{prop}

\begin{proof}
We explicitly find degree $3$ functions in $\Q(X_0^+(N))$ by searching the Riemann-Roch spaces of divisors of the form $P_1+P_2+P_3$, where $P_i\in X_0^+(N)(\Q)$.
\end{proof}

\begin{prop}\label{prop4.14}
The $\Q$-gonality of $X_0(109)$ for $N=109$ is at most $5$.
\end{prop}
\begin{proof}
We construct a function of degree $5$ by looking at the Riemann-Roch spaces of $\Q$-rational divisors of degree $5$ whose support is in the quadratic points obtained by the pullbacks of rational points on $X_0^+(109)$.
\end{proof}

\begin{prop}\label{prop4.28}
The $\Q$-gonality of $X_0(112)$ is at most $6$.
\end{prop}
\begin{proof}
We explicitly find a modular unit of degree $6$ (after 10 hours of computation; see the accompanying \texttt{Magma} code).
\end{proof}

\begin{prop}\label{cor4.16_UB}
$X_0(N)$ has $\Q$-gonality at most $6$ for $N=84, 93, 115, 116, 129, 137, 155, 159$. 
\end{prop}

\begin{proof}
The quotients $X_0^+(N)$ have genus $4$. We find degree $3$ functions in $\Q(X_0^+(N))$ by searching the Riemann-Roch spaces of divisors of the form $P_1+P_2+P_3$, where $P_i\in X_0^+(N)(\Q)$. It follows that $\gon_\Q X_0(N)\leq 2\cdot \gon_\Q X_0^+(N) =6$ by \Cref{prop:UB}.
\end{proof}

\subsection{Upper bounds obtained by considering a dominant map}  Another way of obtaining an upper bound is to explicitly construct a morphism $f:=X_0(N)\rightarrow Y$, where $\gon_\Q Y$ is known. Then $\gon_\Q X_0(N)\leq \deg f \gon_\Q Y$ by \Cref{prop:UB}.
\begin{prop}\label{prop4.5}
The $\Q$-gonality of $X_0(N)$ is at most $4$ for $$N=51,55,56,60,62,63,65,69,75,79,83,89,92,95,101.$$
\end{prop}

\begin{proof}
This is proved in \cite[p.139]{HasegawaShimura_trig}; but as the proof is short and instructive, we repeat it here. By \cite{Bars99} all these curves are bielliptic and have a bielliptic involution $w$ of Atkin-Lehner type. Hence the maps $\pi:X\rightarrow X/w$ are defined over $\Q$, and hence so is the degree $4$ function obtained by composing $\pi$ with a degree $2$ rational function on the elliptic curve $X/w$.
\end{proof}

\begin{prop}\label{prop4.6}
The $\Q$-gonality of $X_0(N)$ is at most $4$ for the following values of $N$, with $Y:=X_0(N)/w_d$:\\
\begin{center}
\begin{tabular}{|c|c|c|c||c|c|c|c|}
  \hline
  $N$ & $g(X_0(N))$ & $d$ & $g(Y)$ & $N$ & $g(X_0(N))$ & $d$ & $g(Y)$\\
  \hline

    $42$ & $5$ & $42$ & $2$ & $77$ & $7$ & $77$ & $2$\\
    $52$ & $5$ & $52$ & $2$ & $80$ & $7$ & $80$ & $2$\\
    $57$ & $5$ & $57$ & $2$ & $87$ & $9$ & $87$ & $2$\\
    $58$ & $6$ & $29$ & $2$ & $91$ & $7$ & $91$ & $2$\\
    $66$ & $9$ & $11$ & $2$ & $98$ & $7$ & $98$ & $2$\\
    $67$ & $5$ & $67$ & $2$ & $100$ & $7$ & $4$ & $2$\\
    $68$ & $7$ & $68$ & $2$ & $103$ & $8$ & $103$ & $2$\\
    $70$ & $9$ & $35$ & $2$ & $107$ & $9$ & $107$ & $2$\\
    $73$ & $5$ & $73$ & $2$ & $121$ & $6$ & $121$ & $2$\\
    $74$ & $8$ & $74$ & $2$ & $125$ & $8$ & $125$ & $2$\\
  \hline
\end{tabular}
\end{center}
\end{prop}

\begin{proof}
This is proved in \cite[p.139]{HasegawaShimura_trig} and the argument of the proof is the same as of \Cref{prop4.5}, with the only difference being that the quotients $X_0(N)/w_d$ are of genus 2 and hence necessarily hyperelliptic.
\end{proof}

Now we produce upper bounds on the $\Q$-gonality by considering the degeneracy maps $X_0(N)\rightarrow X_0(d)$ for $d|N$.

\begin{prop}\label{prop4.9}
The $\Q$-gonality of $X_0(N)$ is bounded from above for the following values of $N$, where $\deg$ denotes the degree of the degeneracy map $X_0(N)\rightarrow X_0(d)$:
\begin{center}
\begin{tabular}{|c|c|c|c||c|c|c|c|}
  \hline
  $N$ & $\gon_{\Q}(X_0(N))\leq$ & $d$ & $\deg$ & $N$ & $\gon_{\Q}(X_0(N))\leq$ & $d$ & $\deg$ \\
  \hline

  $72$ & $4$ & $36$ & $2$ & $144$ & $6$ & $48$ & $3$\\
  $82$ & $6$ & $41$ & $3$ & $148$ & $8$ & $74$ & $2$\\
  $90$ & $6$ & $30$ & $3$ & $150$ & $8$ & $50$ & $4$\\
  $96$ & $4$ & $48$ & $2$ & $156$ & $8$ & $78$ & $2$\\
  $99$ & $6$ & $33$ & $3$ & $160$ & $8$ & $80$ & $2$\\
  $108$ & $6$ & $36$ & $3$ & $175$ & $8$ & $25$ & $8$\\
  $117$ & $6$ & $39$ & $3$ & $176$ & $8$ & $88$ & $2$\\
  $118$ & $6$ & $59$ & $3$ & $184$ & $8$ & $92$ & $2$\\
  $132$ & $8$ & $66$ & $2$ & $192$ & $8$ & $96$ & $2$\\
  $136$ & $8$ & $68$ & $2$ & $196$ & $8$ & $98$ & $2$\\
  $140$ & $8$ & $70$ & $2$ & $200$ & $8$ & $100$ & $2$\\
  \hline
\end{tabular}

\end{center}
\end{prop}
\begin{proof}
There exists a morphism $f:X_0(N)\rightarrow X_0(d)$ of degree $\deg$ over $\Q$ by \Cref{prop:lower_deg}. Therefore, $\gon_{\Q}(X_0(N))\leq \deg\cdot \gon_{\Q}(X_0(d))$.
\end{proof}

Next we obtain upper bounds on $\gon_\Q X_0(N)$ by considering Atkin-Lehner quotients.

\begin{prop}\label{prop4.10}
The $\C$-gonality of $X_0(N)$ is bounded above by $6$ and the $\Q$-gonality is bounded from above for the following values of $N$, with $X:=X_0(N)$ and $Y:=X_0(N)/w_d$:\\
\begin{center}
\begin{tabular}{|c|c|c|c|c||c|c|c|c|c|}
  \hline
  $N$ & $\gon_{\Q}(X)\leq$ & $d$ & $g(Y)$ & $\gon_\Q Y\leq $ &  $N$ & $\gon_{\Q}(X)\leq$ & $d$ & $g(Y)$ & $\gon_\Q Y$ \\
  \hline

  $76$ & $6$ & $76$ & $3$ & $3$& $145$ & $8$ & $29$ & $4$ & $4$\\
  $86$ & $6$ & $86$ & $3$ & $3$& $149$ & $6$ & $149$ & $3$ & $3$\\
  $97$ & $6$ & $97$ & $3$ & $3$ & $151$ & $6$ & $151$ & $3$& $3$\\
  $105$ & $6$ & $35$ & $3$ & $3$ &$161$ & $8$ & $161$ & $4$& $4$\\
  $110$ & $8$ & $55$ & $4$ & $4$ &$169$ & $6$ & $169$ & $3$& $3$\\
  $113$ & $6$ & $113$ & $3$ & $3$ &$173$ & $8$ & $173$ & $4$& $4$\\
  $123$ & $6$ & $41$ & $3$ & $3$ &$177$ & $8$ & $59$ & $4$& $4$\\
  $124$ & $6$ & $31$ & $3$ & $3$ &$179$ & $6$ & $179$ & $3$& $3$\\
  $127$ & $6$ & $127$ & $3$ & $3$ &$188$ & $8$ & $47$ & $4$ & $4$\\
  $128$ & $6$ & $128$ & $3$ & $3$ &$199$ & $8$ & $199$ & $4$ & $4$\\
  $133$ & $8$ & $19$ & $4$ & $4$ &$215$ & $6$ & $215$ & $4$ & $3$\\
  $135$ & $6$ & $135$ & $4$ & $3$ &$239$ & $6$ & $239$ & $3$ & $3$\\
  $139$ & $6$ & $139$ & $3$ & $3$ &$251$ & $8$ & $251$ & $4$ & $4$\\
  $141$ & $6$ & $47$ & $3$ & $3$ &$311$ & $8$ & $311$ & $4$ & $4$\\
  \hline
\end{tabular}
\end{center}
\end{prop}

\begin{proof}
In all the cases above $Y$ is known to not be hyperelliptic. As there exists a morphism of degree $2$ over $\Q$ to $X_0(N)/ w_d $, it follows that $$\gon_{\Q}(X_0(N))\leq 2 \gon_{\Q}(Y) \leq 2g(Y).$$
In the cases $N=135$ and $215$ in the table above where we have $\gon_\Q Y \leq 3 <g(Y)=4$, this was obtained by explicitly computing the trigonal map $Y\rightarrow \PP^1$ and observing that it is defined over $\Q$. 
\end{proof}

\begin{prop}\label{prop4.11}
The $\Q$-gonality of $X_0(N)$ is $\leq8$ for the following values of $N$, with $Y:=X_0(N)/\left<w_{d_1},w_{d_2}\right>$:\\
\begin{center}
\begin{tabular}{|c|c|c||c|c|c|}
  \hline
  $N$ & $d_1$, $d_2$ & $g(Y)$ & $N$ & $d_1$, $d_2$ & $g(Y)$\\
  \hline
    $102$ & $2$, $51$ & $2$ & $171$ & $9$, $19$ & $3$\\
    $106$ & $2$, $53$ & $2$ & $190$ & $19$, $95$ & $2$\\
    $114$ & $3$, $38$ & $2$ & $195$ & $5$, $39$ & $3$\\
    $120$ & $8$, $15$ & $2$ & $205$ & $5$, $41$ & $2$\\
    $126$ & $2$, $63$ & $2$ & $206$ & $2$, $103$ & $2$\\
    $130$ & $10$, $26$ & $2$ & $209$ & $11$, $19$ & $2$\\
    $134$ & $2$, $67$ & $2$ & $213$ & $3$, $71$ & $2$\\
    $138$ & $3$, $69$ & $2$ & $221$ & $13$, $17$ & $2$\\
    $153$ & $9$, $17$ & $2$ & $279$ & $9$, $31$ & $5$\\
    $158$ & $2$, $79$ & $2$ & $284$ & $4$, $71$ & $2$\\
    $165$ & $11$, $15$ & $3$ & $287$ & $7$, $41$ & $2$\\
    $166$ & $2$, $83$ & $2$ & $299$ & $13$, $23$ & $2$\\
    $168$ & $24$, $56$ & $4$ & & &\\
  \hline
\end{tabular}
\end{center}
\end{prop}

\begin{proof}
There exists a morphism of degree $4$ over $\Q$ to $X_0(N)/\left<  w_{d_1}, w_{d_2} \right>$. All these quotients are hyperelliptic by \cite{FurumotoHasegawa1999}. Therefore, $\gon_{\Q}(X_0(N)) \leq 4\cdot2=8$.
\end{proof}

\begin{prop}\label{prop4.13_UB}
The $\Q$-gonality of $X_0(N)$ is at most $6$ for $N$ in the table below, where $Y:=X_0(N)/w_d$. 
\begin{center}
\begin{tabular}{|c|c|c|c||c|c|c|c|}
\hline
$N$ & $g(X_0(N))$ & $d$ & $g(Y)$ & $N$ & $g(X_0(N))$ & $d$ & $g(Y)$\\
  \hline
  
    $105$ & $13$ & $35$ & $3$&
    $147$ & $11$ & $3$ & $5$\\

    $118$ & $14$ & $59$ & $3$&
    $149$ & $12$ & $149$ & $3$\\
    
    $122$ & $14$ & $122$ & $5$&
    $162$ & $16$ & $162$ & $7$\\

    $123$ & $13$ & $41$ & $3$&
    $164$ & $19$ & $164$ & $6$\\

    $124$ & $14$ & $31$ & $3$&
    $181$ & $14$ & $181$ & $5$\\

    $139$ & $11$ & $139$ & $3$&
    $227$ & $19$ & $227$ & $5$\\

    $141$ & $15$ & $47$ & $3$&
    $239$ & $20$ & $239$ & $3$\\

    $146$ & $17$ & $146$ & $5$ & & & &\\

    \hline
\end{tabular}
\end{center}
\end{prop}
\begin{proof}
For $N=105,118,123,124,139,141,149,239$ the quotients $Y=X_0(N)/w_d$ are trigonal over $\Q$ since they are of genus $3$. For $N=122,146,147,162,164,181,227$, the quotients are trigonal over $\Q$ since they are trigonal over $\C$ of genus $\geq5$ by \cite{HasegawaShimura1999} and we can apply \Cref{Jeon-Park} (i). It follows that $\gon_\Q X_0(N)\leq 6$. 
\end{proof}

\begin{prop}\label{tetragonal_quotients}
    $X_0(N)$ has $\Q$-gonality at most $8$ for
    $$N=152,157,163,183,185,197,203,211,223,263,269,359.$$
\end{prop}

\begin{proof}
    The quotients $X_0^+(N)$ have genus $5$ or $6$ and are not trigonal by \cite{HasegawaShimura1999}. We explicitly find degree $4$ functions in $\Q(X_0^+(N))$ using the \texttt{Magma} functions Genus5GonalMap(C) and Genus6GonalMap(C).  It follows that $\gon_\Q X_0(N)\leq 2\cdot \gon_\Q X_0^+(N)\leq 8$.
\end{proof}

\subsection{Lower bounds obtained by reduction modulo $p$}

As mentioned in \Cref{sec:lb}, an important technique for obtaining a lower bound for the $\Q$-gonality is by computing the $\F_p$-gonality. We will use certain tricks to greatly reduce the computational time needed to give a lower bound for the $\F_p$-gonality. The following two propositions explain how we do this in more detail.

\begin{prop}\label{prop4.26}
The $\Q$-gonality of $X_0(N)$ for $N=99$ is at least $6$.
\end{prop}
\begin{proof}
Let $X:=X_0(99)$. We compute $\#X(\F_5)=6$. Suppose $f$ is an $\F_5$-rational function of degree $\leq 5$. By the pigeonhole principle (as in \Cref{lem:F_p_div_search}), it follows that either there is a point $c\in \PP^1(\F_5)$ such that $f^{-1}(c)$ contains no $\F_5$-rational points, or $\#f^{-1}(c)(\F_5)=1$ for every $c\in \PP^1(\F_5)$.

Suppose the former and let $g(x):=1/(f(x)-c)$. Hence $g^{-1}(\infty)$ has no $\F_5$-rational points. Hence $g$ lies in the Riemann-Roch space of a divisor of one of the following forms: $D_5$, $D_4$, $D_3+D_2$ or $D_2+D_2'$, where $D_i$ is an irreducible $\F_5$-rational effective divisor of degree $i$. Searching among the Riemann-Roch spaces of such divisors, we find that there are no non-constant functions.

Suppose now the latter. Now we can fix a $P\in X(\F_5)$ and suppose without loss of generality that $g^{-1}(\infty)(\F_5)={P}$. Hence $g$ will be found in the Riemann-Roch spaces of $P+D_4$, $P+D_2+D_2'$ or $P+D_3$, where the notation is as before. Searching among the Riemann-Roch spaces of such divisors, we find that there are no non-constant functions.

\end{proof}

\begin{prop}\label{prop:130}
The $\Q$-gonality of $X_0(N)$ for $N=130$ is at least $8$.
\end{prop}
\begin{proof}
Let $X:=X_0(130)$. We compute $\#X(\F_3)=8$. Suppose $f$ is an $\F_3$-rational function of degree $\leq 7$. By the pigeonhole principle (as in \Cref{lem:F_p_div_search}), it follows that either there is a point $c\in \PP^1(\F_3)$ such that $f^{-1}(c)(\F_3)$ contains at most one $\F_3$-rational point or $\#f^{-1}(c)(\F_3)=2$ for every $c\in \PP^1(\F_3)$.

Suppose the former and let $g(x):=1/(f(x)-c)$. Hence $g^{-1}(\infty)$ has one $\F_3$-rational point. Hence $f$ lies in the Riemann-Roch space of an effective degree $7$ divisor supported on at most $1$ rational point.
Searching among the Riemann-Roch spaces of such divisors, we find that there are no non-constant functions.

Suppose now the latter. Now we can fix a $P\in X(\F_3)$ and suppose without loss of generality that $g^{-1}(\infty)(\F_3)=\{P, Q\}$ for some $Q\in X(\F_3)$. Hence $g$ will be found in the Riemann-Roch space of an effective degree $7$ divisor for which the set of rational points in the support is exactly $\{P,Q\}$, with $Q$ varying through all $Q\in X(\F_3)$. Searching among the Riemann-Roch spaces of such divisors, we find that there are no non-constant functions. 
\end{proof}

We apply a similar approach by producing a lower bound for the $\F_p$-gonality to obtain a lower bound for the $\Q$-gonality for a large number of $N$.

\begin{prop}\label{prop:fp}
A lower bound (LB) for the $\Q$-gonality of $X_0(N)$ is given in the following table, where $p$ is a prime of good reduction for $X_0(N)$:\\

\begin{center}
\begin{tabular}{|c|c|c|c||c|c|c|c||c|c|c|c||c|c|c|c|}
\hline
$N$ & \textup{LB} & $p$ & \textup{time} & $N$ & \textup{LB} & $p$ & \textup{time} & $N$ & \textup{LB} & $p$ & \textup{time} & $N$ & \textup{LB} & $p$ & \textup{time}\\
  \hline
$38$ & $4$ & $5$ & $2$ \textup{sec} &
$115$ & $6$ & $3$ & $56$ \textup{sec} &
$151$ & $6$ & $5$ & $94$ \textup{sec} &
$181$ & $6$ & $3$ & $9$ \textup{sec} \\

$44$ & $4$ & $5$ & $4$ \textup{sec} &
$116$ & $6$ & $3$ & $10$ \textup{sec} &
$152$ & $8$ & $3$ & $20.5$ \textup{min} &
$187$ & $8$ & $2$ & $1.5$ \textup{hrs} \\

$53$ & $4$ & $5$ & $9$ \textup{sec} &
$117$ & $6$ & $5$ & $10$ \textup{sec} &
$153$ & $8$ & $5$ & $2$ \textup{hrs} &
$189$ & $8$ & $2$ & $3$ \textup{min} \\

$61$ & $4$ & $3$ & $1$ \textup{sec} &
$118$ & $6$ & $3$ & $12$ \textup{sec} &
$154$ & $8$ & $5$ & $2$ \textup{days} &
$192$ & $8$ & $5$ & $4$ \textup{days} \\

$76$ & $6$ & $5$ & $8$ \textup{sec} &
$122$ & $6$ & $3$ & $55$ \textup{sec} &
$157$ & $8$ & $3$ & $37$ \textup{sec} &
$193$ & $6$ & $3$ & $28$ \textup{sec} \\

$82$ & $6$ & $5$ & $62$ \textup{sec} &
$127$ & $6$ & $3$ & $24$ \textup{sec} &
$160$ & $8$ & $7$ & $173$ \textup{sec} &
$196$ & $8$ & $5$ & $2.9$ \textup{hrs} \\

$84$ & $6$ & $5$ & $67$ \textup{min} &
$128$ & $6$ & $3$ & $4$ \textup{sec} &
$162$ & $6$ & $5$ & $53$ \textup{sec} &
$197$ & $6$ & $3$ & $36$ \textup{min} \\

$86$ & $6$ & $3$ & $135$ \textup{sec} &
$130$ & $8$ & $3$ & $20$ \textup{min} &
$163$ & $7$ & $5$ & $3$ \textup{min} &
$198$ & $8$ & $5$ & $7$ \textup{days} \\

$93$ & $6$ & $5$ & $4$ \textup{sec} &
$132$ & $8$ & $6$ & $22.2$ \textup{hrs} &
$169$ & $6$ & $5$ & $2.5$ \textup{min} &
$200$ & $8$ & $3$ & $1.6$ \textup{hrs} \\

$99$ & $6$ & $5$ & $94$ \textup{sec} &
$134$ & $8$ & $3$ & $1.2$ \textup{hrs} &
$170$ & $8$ & $3$ & $100.5$ \textup{hrs} &
$201$ & $8$ & $2$ & $4$ \textup{hrs} \\

$102$ & $8$ & $5$ & $3.3$ \textup{hrs} &
$136$ & $8$ & $5$ & $13.4$ \textup{hrs} &
$172$ & $8$ & $3$ & $3.3$ \textup{hrs} &
$217$ & $8$ & $2$ & $2$ \textup{min} \\

$106$ & $8$ & $7$ & $35$ \textup{hrs} &
$137$ & $6$ & $3$ & $4$ \textup{sec} &
$175$ & $2$ & $2$ & $18.7$ \textup{sec} &
$229$ & $8$ & $3$ & $6.5$ \textup{min} \\

$108$ & $6$ & $5$ & $27$ \textup{min} &
$140$ & $8$ & $3$ & $25.6$ \textup{hrs} &
$176$ & $8$ & $3$ & $12$ \textup{min} &
$233$ & $8$ & $2$ & $2$ \textup{hrs} \\

$109$ & $5$ & $3$ & $83$ \textup{sec} &
$144$ & $6$ & $5$ & $56$ \textup{sec} &
$178$ & $8$ & $3$ & $4.5$ \textup{hrs} &
$241$ & $8$ & $2$ & $2.5$ \textup{min} \\

$112$ & $6$ & $3$ & $10$ \textup{hrs} &
$147$ & $6$ & $5$ & $4.5$ \textup{min} &
$179$ & $6$ & $5$ & $10$ \textup{min} &
$247$ & $8$ & $2$ & $4$ \textup{hrs} \\

$113$ & $6$ & $3$ & $4$ \textup{sec} &
$148$ & $8$ & $6$ & $3.2$ \textup{hrs} &
$180$ & $7$ & $7$ & $9$ \textup{days} &
$277$ & $8$ & $6$ & $7$ \textup{days} \\

$114$ & $8$ & $5$ & $53.2$ \textup{hrs} &
$150$ & $8$ & $7$ & $34.5$ \textup{hrs} & & & & & & & & \\

    \hline
\end{tabular}
\end{center}
\end{prop}
\smallskip

\begin{proof}
In all the cases we compute that there are no functions of degree $<d$ in $\F_p(X_0(N))$, where $p$ is as listed in the table.  In computationally more demanding cases, i.e. when $d$, $p$ and the genus of $X_0(N)$ are larger, we use techniques as in Propositions \ref{prop4.26} and \ref{prop:130}. 
All the \texttt{Magma} computations proving this can be found in our repository.
\end{proof}

\begin{prop}\label{prop:prosla}
The following genus $4$ quotients $X_0(N)/w_d$ are \textbf{not} trigonal over $\Q$, where $p$ is a prime of good reduction for $X_0(N)$:\\
\begin{center}
\begin{tabular}{|c|c|c||c|c|c|}
\hline
$N$ &  $d$ &  $p$ & $N$ &  $d$ &  $p$\\
  \hline
$110$ & $55$ & $7$ & 
$188$ & $47$& $3$\\

$145$ & $29$ & $11$&
$199$ & $199$& $5$\\

$161$ & $161$& $5$ &
$251$ & $251$ & $3$\\

$173$ & $173$& $5$ &
$311$ & $311$ & $5$\\

$177$ & $59$& $5$ & & &\\

    \hline
\end{tabular}
\end{center}
\end{prop}

\begin{proof}
We find that the quotients $X_0(N)/w_d$ have no functions of degree $3$ over $\F_p$ for $d$ and $p$ listed in the table above. 
\end{proof}

\begin{remark}
It is worth mentioning that Bars and Dalal \cite{BarsDalal22} determined all quotients $X_0^+(N)$ which are trigonal over $\Q$, thereby independently obtaining the results for $N=161,173,199,251,311$ in the above proposition. 
\end{remark}

\subsection{Lower bounds obtained by the Castelnuovo-Severi inequality}

\begin{prop}\label{prop4.13}
The $\Q$-gonality and $\C$-gonality of $X_0(N)$ are at least $6$ for $N$ in the table below, where $Y:=X_0(N)/w_d$. 
\begin{center}
\begin{tabular}{|c|c|c|c||c|c|c|c|}
\hline
$N$ & $g(X_0(N))$ & $d$ & $g(Y)$ & $N$ & $g(X_0(N))$ & $d$ & $g(Y)$\\
  \hline
  
    $105$ & $13$ & $35$ & $3$&
    $141$ & $15$ & $47$ & $3$\\
    
    $116$ & $13$ & $116$ & $4$&
    $146$ & $17$ & $146$ & $5$ \\  
    
    $118$ & $14$ & $59$ & $3$&
    $149$ & $12$ & $149$ & $3$\\
    
    $123$ & $13$ & $41$ & $3$&
    $164$ & $19$ & $164$ & $6$\\
    
    $124$ & $14$ & $31$ & $3$&
    $227$ & $19$ & $227$ & $5$\\

    $139$ & $11$ & $139$ & $3$&
    $239$ & $20$ & $239$ & $3$\\

    \hline
\end{tabular}
\end{center}
\end{prop}

\medskip
\begin{proof}

By \cite{JeonPark05} it follows that $\gon_\C X_0(N)\geq 5$ and then \Cref{cor:CS} gives $\gon_\C X_0(N) \geq 6$. 
\end{proof}

\begin{cor}\label{cor4.18}
The $\Q$-gonality is at least $8$ and the $\C$-gonality is at least $6$ for $N=110,161,173,177,188,199,251,311$.
\end{cor}

\begin{proof}
Since all the curves have genus $4$ quotients $X_0(N)/w_d$ as in \Cref{prop:prosla} and genus $\geq 14$, applying \Cref{cor:CS} proves $\gon_\C\geq 6$.

For $N\neq173$, since the quotients are not trigonal over $\Q$, we can apply \Cref{cor:CS} to prove $\gon_\Q X_0(N)\geq8$.

For $N=173$, since the quotient is not trigonal over $\Q$, we can apply the Castelnuovo-Severi inequality similarly as in \Cref{cor:CS} and conclude that $\gon_\Q(X_0(173))\geq7$. Now we explicitly compute that there are no degree $7$ functions over $\F_3$ as in \Cref{prop:fp} (the computation takes 26 seconds).
\end{proof}

\begin{prop}\label{prop4.20}
The $\Q$-gonality and $\C$-gonality are at least $8$ for the following $N$, where $Y:=X_0(N)/w_d$:\\
\begin{center}
\begin{tabular}{|c|c|c|c||c|c|c|c|}
\hline
$N$ &  $g(X_0(N))$ & $d$ & $g(Y)$ & $N$ &  $g(X_0(N))$ & $d$ & $g(Y)$\\
  \hline
    $120$ & $17$ & $15$ & $5$ &
    $203$ & $19$ & $203$ & $6$\\
    
    $126$ & $17$ & $63$ & $5$ &
    $205$ & $19$ & $41$ & $6$\\

    $138$ & $21$ & $23$ & $5$ &
    $206$ & $25$ & $206$ & $8$\\

    $156$ & $23$ & $39$ & $6$ &
    $209$ & $19$ & $209$ & $5$\\

    $158$ & $19$ & $79$ & $5$ &
    $213$ & $23$ & $71$ & $5$\\

    $165$ & $21$ & $11$ & $7$ &
    $221$ & $19$ & $221$ & $6$\\
    
    $166$ & $20$ & $83$ & $6$ &
    $263$ & $22$ & $263$ & $5$\\
    
    $168$ & $25$ & $56$ & $9$ &
    $269$ & $22$ & $269$ & $6$\\
    
    $171$ & $17$ & $171$ & $5$ &
    $279$ & $29$ & $279$ & $9$\\
    
    $183$ & $19$ & $183$ & $6$ &
    $284$ & $34$ & $71$ & $7$ \\
    
    $184$ & $21$ & $23$ & $5$ &
    $287$ & $27$ & $287$ & $7$\\
    
    $185$ & $17$ & $185$ & $5$ &
    $299$ & $27$ & $299$ & $6$\\
    
    $190$ & $27$ & $95$ & $6$ &
    $359$ & $30$ & $359$ & $6$\\
        
    $195$ & $25$ & $39$ & $9$ & & & & \\
    
    \hline
\end{tabular}
\end{center}
\end{prop}

\begin{proof}
Since the quotients $Y=X_0(N)/w_d$ are not trigonal over $\C$ \cite{HasegawaShimura1999}, they are not trigonal over $\Q$ either, and by applying \Cref{cor:CS} the result follows.
\end{proof}

\begin{prop}\label{Qgon10}
    The $\Q$-gonality of $X_0(N)$ for $N=271$ is $10$ and the $\C$-gonality is $8$.
\end{prop}
\begin{proof}
    The quotient $X_0^+(271)$ is of genus $6$ and pentagonal over $\Q$ (we found a function of degree $5$ using the \texttt{Magma} function Genus6GonalMap(C)). Further, the quotient is not tetragonal over $\Q$ since it is not tetragonal over $\F_3$, but it is tetragonal over $\C$ because of \Cref{prop_A.1}. It now follows from \Cref{cor:CS} that there are no functions $f:X_0(271)\to\mathbb{P}^1$ of degree $\leq7$ defined over $\C$ and that there are no functions $f:X_0(271)\to\mathbb{P}^1$ of degree $\leq9$ defined over $\Q$.
\end{proof}

\begin{prop}\label{Cgon6cs}
The $\C$-gonality for is bounded from below for the following values of $N$, where $Y:=X_0(N)/w_d$: 
\begin{center}
\begin{tabular}{|c|c|c|c|c||c|c|c|c|c|}
\hline
$N$ &  $g(X_0(N))$ & $d$ & $g(Y)$ & $\textup{gon}_\C\geq$ & $N$ &  $g(X_0(N))$ & $d$ & $g(Y)$ & $\textup{gon}_\C\geq$\\
\hline
    $102$ & $15$ & $51$ & $5$ & $6$ &
    $202$ & $24$ & $101$ & $9$ & $7$ \\

    $129$ & $13$ & $129$ & $4$ & $6$ &
    $204$ & $31$ & $68$ & $12$ & $8$ \\

    $150$ & $19$ & $75$ & $7$ & $6$ &
    $210$ & $41$ & $35$ & $15$ & $8$ \\

    $152$ & $17$ & $152$ & $6$ & $6$ &
    $211$ & $17$ & $211$ & $6$ & $6$ \\

    $155$ & $15$ & $155$ & $4$ & $6$ &
    $214$ & $26$ & $107$ & $9$ & $8$ \\

    $159$ & $17$ & $159$ & $4$ & $6$ &
    $219$ & $23$ & $219$ & $8$ & $7$ \\

    $174$ & $27$ & $87$ & $8$ & $8$ &
    $223$ & $18$ & $223$ & $6$ & $7$ \\

    $175$ & $15$ & $175$ & $5$ & $6$ &
    $257$ & $21$ & $257$ & $7$ & $7$ \\

    $186$ & $29$ & $62$ & $11$ & $8$ &
    $281$ & $23$ & $281$ & $7$ & $8$ \\

    $194$ & $23$ & $97$ & $7$ & $8$ &
    $293$ & $24$ & $293$ & $8$ & $8$ \\

\hline
\end{tabular}
\end{center}
\end{prop}

\begin{proof}
    Using the degree $2$ maps to $X_0(N)/ w_d $, we can apply the Castelnuovo-Severi inequality similarly as in \Cref{cor:CS} and get the lower bound for $\gon_\C(X_0(N))$.
\end{proof}

\subsection{Lower bounds by Green's conjecture}

Finally, we prove the lower bound for $\gon_\C X_0(N)$ where we can. We first recall the following result for completeness.

\begin{prop}\label{prop:gon_ge_6}
The $\C$-gonality is at least $6$ for 
\begin{align*}N=&114,132,134,135,140,145,150,151,152,160,165,166,168,170,171,172,\\
& 174,175,176,178,182,183,185,186,189,192,194,195,196 \textup{ and } N \geq 198.\end{align*}
\end{prop}
\begin{proof}
This is \cite[Proposition 4.4]{HasegawaShimura_trig}.
\end{proof}

\begin{prop}\label{prop4.24}
The $\C$-gonality of $X_0(90)$ is at least $6$.
\end{prop}
\begin{proof}
Let $X:=X_0(90)$. We first note that the degree $3$ map to $X_0(30)$ gives an upper bound on the $\Q$-gonality by \Cref{prop:UB}. From \cite[Theorem 0.1]{JeonPark05}, it follows that $\gon_\C X>4$. By \cite[Table 1]{JeonKimPark06}, we see that $\beta_{3,2}=0$, and hence we conclude by \Cref{nog51} that $X$ has no $g_5^1.$ It follows that $\gon_\C X \geq 6$ and hence $\gon_\C X=\gon_\Q X =6.$
\end{proof}

\begin{prop} \label{Cgon6}
The $\C$-gonality for $N=84,86,93,106,115,127,128,133,137$ is greater or equal to $6$. 
\end{prop}
\begin{proof}
By \cite{JeonPark05}, $\gon_\C X_0(N)\geq 5$. By \cite[Table 1]{JeonKimPark06}, we see that $\beta_{3,2}=0$ for all $N \neq 86,127$, while for $N=86, 127$ we compute $\beta_{2,2}=\beta_{3,2}=0$ in \texttt{Magma} (the computations take 3.5 and 1.5 hours, respectively). We conclude by \Cref{nog51} that $X_0(N)$ has no $g_5^1,$ hence $\gon_\C X_0(N) \geq 6$.
\end{proof}

\section{Mordell-Weil sieving on Brill-Noether varieties}
\label{sec:MW}
The only cases that remain unsolved for $N<145$ are $N=97$ and $133$ for which we have $5 \leq \gon_\Q X_0(97)\leq 6$ and $\gon_\Q X_0(133)\leq 8$. In this section we show that the upper bound is correct in both cases and also prove that $7 \leq \gon_\Q X_0(145)$.

For a curve $X$, denote by $W_d^r(X)$ the closed subvariety of $\Pic^d(X)$ classifying divisor classes of effective divisors $D$ of degree $d$ which are contained in linear systems of dimension $\geq r$, or equivalently, for which $\ell(D)\geq r+1$. Obviously a curve $X$ with $X(\Q)\neq \emptyset$ has a function of degree $d$ over $\Q$ if and only if $W_d^1(X)(\Q)\neq \emptyset$. Let $X:=X_0(N)$ and $\mu:\Pic^d X \rightarrow J_0(N)$ the map defined by $\mu(D):=D-w(D)$. We obviously have that the $\mu(J_0(N))\subset J_0(N)^-$. For $N=97$ and $133$, we compute that $J_0(N)^-(\Q)$ is of rank 0 by computing that its analytic rank is 0 (see e.g. \cite[Section 3]{Deg3Class}).

Suppose $D\in W_d^1(X)(\Q)$ and $p>2$ is a prime of good reduction for $X$. We have the following commutative diagram
$$
\xymatrix{
& W_d^1(X)(\Q) \ar[r]^{\mu} \ar[d]^{\red_p} & J_0(N)^-(\Q) \ar[d]^{\red_p} \\
& W_d^1(X)(\F_p) \ar[r]^{\mu} & J_0(N)^-(\F_p)}, 
$$
where the vertical maps are reduction modulo $p$. Suppose now that there exists a $D\in W_d^1(X)(\Q)$. Then $\mu(D)$ lies in $\red_p^{-1}(\mu(W_d^1(X)(\F_p)))$. The set $W_d^1(X)(\F_p)$ can be computed in practice by simply finding all the effective degree $d$ divisors whose Riemann-Roch spaces have dimension $\geq 2$. Note that in our cases $J_0(N)^-(\Q)$ is a torsion group and $\red_p$ is injective on the torsion of $J_0(N)(\Q)$ \cite[Appendix]{katz81}. The same procedure can be applied for a set $S$ of multiple primes $p>2$ of good reduction, in which case we get 
$$\mu(D)\in \bigcap _{p\in S}\red_p^{-1}(\mu(W_d^1(X)(\F_p))).$$
If $$\bigcap _{p\in S}\red_p^{-1}(\mu(W_d^1(X)(\F_p)))=\emptyset$$ it follows that $W_d^1(X)(\Q)=\emptyset$ and indeed this is what we will show. In our cases it will be enough to take $S$ consisting of a single prime.
\begin{prop} \label{prop:97_133} The $\Q$-gonality of $X_0(97)$ is $6$ and the $\Q$-gonality of $X_0(133)$ is $8$. The $\Q$-gonality of $X_0(145)$ is $\geq7$.
\end{prop}
\begin{proof}
By \cite[Theorem 4]{mazur77} we know that $J_0(97)^-(\Q)\simeq \Z/8\Z$ and is generated by $D_0=[0-\infty]$, where $0$ and $\infty$ are the two cusps of $X_0(97)$. We compute 
$$\red_3^{-1}(\mu(W_5^1(X_0(97)(\F_3))))=\{0\}, \quad \red_5^{-1}(\mu(W_5^1(X_0(97)(\F_5))))=\{D_0, 7D_0\},$$ 
$$\quad \red_7^{-1}(\mu(W_5^1(X_0(97)(\F_7))))=\emptyset.$$
Hence sieving with either $\{3,5\}$ or just the prime $7$ proves that $W_5^1(X_0(97))(\Q)=\emptyset$. Hence it follows that $X_0(97)$ is of gonality $6$ over $\Q$.

The cases of $X_0(133)$ and $X_0(145)$ are more involved than $X_0(97)$ because we cannot compute the torsion group exactly. The rank of $J_0(N)^-(\Q)$ is 0 in both cases, so $J_0(N)^-(\Q)$ is contained in $J_0(N)(\Q)_{\tors}$.

First we solve the case $N=133$. Using the methods of \cite[Section 4]{Deg3Class}, we obtain that $J_0(N)(\Q)_{\tors}$ is isomorphic to a subgroup of $\Z/6\Z \times \Z/ 360 \Z$. We find cuspidal divisors (i.e. divisors supported on cusps of $X_0(N)$) $A,B$ which generate a subgroup $T:=\diam{A,B} \simeq \Z/6 \Z \times \Z/ 180 \Z$. Thus it follows that for any $x\in J_0(N)^-(\Q)$, we have $2x \in T$. Hence we use the map $2\mu$, sending a divisor $D$ to $2(D-w(D))$ instead of $\mu$ (which we used for $X_0(97)$). For sieving, we will just need to use the prime 3. We observe that $\#X_0(133)(\F_3)=8$, so if there exists a function of degree $7$ on $X_0(133)$ over $\Q$, then there has to exist a function of degree $7$ over $\Q$ whose reduction modulo $3$ has a polar divisor that is supported on at most 2 $\F_3$-rational points, using the same arguments as in \Cref{lem:F_p_div_search}. Thus we need only search effective divisors supported on at most 2 $\F_3$-rational points; if $R\subset W_7^1(X_0(133))(\F_3)$ is the set of all divisors classes in $W_7^1(X_0(133))(\F_3)$ represented by divisors supported on at most 2 $\F_3$-rational points, then 
$$\red_3^{-1}(2\cdot\mu(R))=\emptyset \quad \text{ implies }\quad \red_3^{-1}(2\cdot\mu(W_7^1(X_0(133))(\F_3)))=\emptyset.$$
This is exactly what we obtain, proving the result. 

Now we solve the case $N=145$. Using the methods of \cite[Section 4]{Deg3Class} again, we obtain that $J_0(N)(\Q)_{\tors}$ is isomorphic to a subgroup of $\Z/2\Z \times\Z/2\Z \times\Z/2\Z \times\Z/14\Z \times \Z/ 140 \Z$. We find cuspidal divisors $A,B$ which generate a subgroup $T:=\diam{A,B} \simeq \Z/14 \Z \times \Z/ 140 \Z$. Hence we use the map $2\mu$ as in the case $N=133$. For sieving, we will again just need to use the prime 3. Using the same techniques as for the $N=133$ we obtain 
$$\red_3^{-1}(2\cdot\mu(R))=\emptyset \quad \text{ implies }\quad \red_3^{-1}(2\cdot\mu(W_6^1(X_0(145))(\F_3)))=\emptyset$$
which proves that the $\Q$-gonality of $X_0(145)$ is $\geq7$ as desired.
\end{proof}

For $N=97$, the program described in \Cref{prop:97_133} terminates after $7.6$ minutes, for $N=133$ after $6.3$ hours, and for $N=145$ after $1.6$ minutes.

\section{Proofs of the Main Results}
First observe that \Cref{main_tm} follows from the fact that the upper and lower bounds populating the tables agree. We now prove Theorems \ref{thm:tetragonal}, \ref{thm:pentagonal} and \ref{thm:hexagonal} in separate subsections. Before proceeding with the proofs, recall that Ogg \cite{Ogg74} determined the hyperelliptic curves $X_0(N)$ and Hasegawa and Shimura \cite{HasegawaShimura_trig} determined all $X_0(N)$ that are trigonal over $\Q$.
\subsection{Tetragonal curves}
\begin{proof}[Proof of \Cref{thm:tetragonal}]
By \cite[Prosposition 4.4]{HasegawaShimura_trig} the $\C$-gonality (and therefore the $\Q$-gonality) of $X_0(N)$ is $\geq5$ for $N\geq 192$. Hence, we only need to consider the $N\leq 191$ such that $X_0(N)$ is not of gonality $\leq 3$ over $\Q$ and such that the gonality over $\C$ is not $\geq 5$. For these values, the results follow from \Cref{main_tm}.
\end{proof}
\subsection{Pentagonal curves}

\begin{proof}[Proof of \Cref{thm:pentagonal}]
In \Cref{thm:tetragonal} we determined all the cuvres $X_0(N)$ that are tetragonal over $\Q$. \Cref{prop4.14} and \Cref{prop:fp} tell us that for $N=109$ the curve is pentagonal over $\Q$. 

Hasegawa and Shimura \cite[Proposition 4.4]{HasegawaShimura_trig} have proved that the $\C$-gonality (and therefore the $\Q$-gonality) is $\geq6$ for $N\geq 198$. In \Cref{main_tm} we proved that for the remaining (i.e. those not of $\Q$-gonality $\leq 4$) $X_0(N)$ with $N\leq 197,$ the $\Q$-gonality is $\geq6$. Hence it follows that for $N\neq 109$ the $\Q$-gonality is either smaller or larger than $5$, proving the result. 
\end{proof}

\subsection{Hexagonal curves}
In this section we prove \Cref{thm:hexagonal} by showing that there are no curves $X_0(N)$ which are hexagonal over $\Q$ besides those that we will list in \Cref{thm:hexagonal}. An important tool used here is the following inequality of Ogg. It originally appeared as \cite[Theorem 3.1]{Ogg74}, but we state it in the simpler form as in \cite[Lemma 3.1]{HasegawaShimura_trig}.

\begin{lem}
    For a prime $p\nmid N$, let
    $$L_p(N):=\frac{p-1}{12}\psi(N)+2^{\omega(N)},$$
    where $\psi(N)=N\prod_{q\mid N}(1+\frac{1}{q})$ and $\omega(N)$ is the number of distinct prime divisors of $N$. Then
    $$L_p(N)\leq\#{X}_0(N)(\F_{p^2}).$$
\end{lem}

This allows us to get a finite and moreover a relatively small list of possible values $N$ for which $X_0(N)$ is hexagonal.

\begin{prop}\label{oggineq}
    The curve $X_0(N)$ is not hexagonal over $\Q$ for $N>335$ and 
    \begin{align*}
        N\in
        \{&220,222,224-226,228,230-232,234,236-238,242,244-246,248,250,252,\\
        &254-256,258,260-262,264-268,270,272-276,278,280,282,285,286,288,\\
        &290,292,294-298,300-306,308-310,312,314-316,318-330,332-335\}.
    \end{align*}
\end{prop}
\begin{proof}
If the curve $X_0(N)$ is hexagonal over $\Q$, we must have $\#{X}_0(N)(\F_{p^2})\leq 6(p^2+1)$ by setting $q=p^2$ in \Cref{lem:fp}. Therefore, the inequality $L_p(N)\leq 6(p^2+1)$ must hold for every prime $p\nmid N$. Now using the same technique as in \cite[Lemma 3.2.]{HasegawaShimura_trig} we complete the proof.
\end{proof}

\begin{prop}\label{fp2points}
    The curve $X_0(N)$ is not hexagonal over $\Q$ for the following $N$:
\begin{center}
\begin{tabular}{|c|c|c||c|c|c|}
\hline
$N$ & $p$ & $\#{X}_0(N)(\F_{p^2})$ & $N$ & $p$ & $\#{X}_0(N)(\F_{p^2})$\\
\hline
    $182$ & $3$ & $64$ &
    $253$ & $2$ & $32$ \\
    $207$ & $2$ & $32$ &
    $259$ & $2$ & $34$ \\
    $208$ & $3$ & $68$ &
    $283$ & $3$ & $64$ \\
    $216$ & $5$ & $168$ &
    $289$ & $2$ & $32$ \\
    $218$ & $3$ & $64$ &
    $307$ & $3$ & $68$ \\
    $235$ & $2$ & $32$ &
    $313$ & $3$ & $68$ \\
    $240$ & $7$ & $312$ & 
    $317$ & $2$ & $35$ \\
    $243$ & $2$ & $33$ &
    $331$ & $2$ & $37$ \\
\hline
\end{tabular}
\end{center}
\end{prop}

\begin{proof}
    For each of these $N$ we have that $\#{X}_0(N)(\F_{p^2})> 6(p^2+1)$ from which it follows by \Cref{lem:fp} that $X_0(N)$ is not hexagonal over $\Q$.
\end{proof}

\begin{prop}\label{prop212}
    The curve $X_0(212)$ is not hexagonal over $\Q$.
\end{prop}

\begin{proof}
    The curve $X_0(106)$ has $\Q$-gonality equal to $8$ by \Cref{prop4.11} and \Cref{prop:fp}. Therefore, the $\Q$-gonality of the curve $X_0(212)$ must be at least $8$ by \Cref{poonenlowerbound}.
\end{proof}
\begin{proof}[Proof of \Cref{thm:hexagonal}]
    This now follows directly from \Cref{main_tm} and Propositions \ref{oggineq} - \ref{prop212}.
\end{proof}





\section{Limits of our methods and future work} \label{sec:future work}
In this section we briefly discuss why our methods can't be pushed much further for $X_0(N)$, and how they could be applied to other modular curves, or even more generally, to arbitrary curves. 

\subsection{Complexity and obstacles to pushing further}
\smallskip

It is a natural question what stopped us from going further, i.e. determining the gonality for larger $N$. Unfortunately, as $N$ gets larger, computations get much harder. As the genus of $X_0(N)$ becomes larger, computing models, their quotients, and computations in Riemann-Roch spaces over $\Q$ all become much more difficult. Furthermore, as the gonalities get larger, the degrees of divisors and the sheer number of divisors needed to be considered (as in \Cref{prop:fp}) makes computations of the $\F_p$-gonality far more difficult. In particular, the most computationally demanding computations that we do are the ones to determine a lower bound for the $\F_p$-gonality. This requires computing the dimension of a huge number of Riemann-Roch spaces. While the complexity of computing a single Riemann-Roch space is polynomial in the size of the input (see \cite{Hess02}), the number of Riemann-Roch spaces that need to be computed to give a lower bound of $d$ for the $\F_p$-gonality is $O(p^d)$. By \cite[Theorem 0.1]{abramovich} we can expect the gonality of $X_0(N)$ to grow linearly in $N$, which suggests that the number of Riemann-Roch spaces that need to be checked grows exponentially with $N$ (and doubly exponentially with the size of $N$). The necessity of choosing (very) small $p$ when computing the $\F_p$-gonality is clear from the complexity discussion above.   

It should be clear that our methods do not give an algorithm for computing the gonality of $X_0(N)$. They produce a lower bound and an upper bound, but there is no guarantee that they will be equal. In practice this (the bounds not matching) is exactly what happens for $N$ larger than the ones that we list. It often happens that for a curve X, one has $\gon_{\F_p} X< \gon_{\Q}X$. For example this happens when $n=\gon_\C X<\gon_{\Q}X$, and a degree $n$ map to $\PP^1$ is defined over a number field $K$ in which $p$ splits completely. Then it follows that $\gon_{\F_p} X\leq n$, and hence the lower bound obtained by computing  $\gon_{\F_p} X$ will not be sharp. 

In practice, for the $N$-s where we started encountering difficulties and were unable to compute the exact gonality, the bounds not matching was the more common problem than the computations being too demanding.

\subsection{Applications and future work}

As has been mentioned in the introduction, the LMFDB will contain modular curves $X_\Gamma$ for all congruence subgroups $\Gamma\leq \SL_2(\Z)$ up to some level. Since the methods we use to obtain lower and upper bounds for the gonality in \Cref{sec:lb} and \Cref{sec:ub}, respectively, work for any curve, the same methods should be, in principle, able to determine the gonality of many other $X_\Gamma$. However, there are several properties of $X_0(N)$ that make them particularly amenable to our methods. The first is the existence of Atkin-Lehner involutions on $X_0(N)$. Their existence makes us, on one hand able to construct explicit functions, giving upper bounds, and on the other hand to use the Castelnuovo-Severi inequality to obtain lower bounds. Another useful property of $X_0(N)$ is the existence of rational points, namely cusps. This is a necessary assumption in some of the methods that we use, e.g. \Cref{tm:TT} and \Cref{cor:TT}. While other modular curves have rational cusps, many interesting ones, such as $X_{ns}^+(p)$ (the modular curves corresponding to the normalizer of the non-split Cartan subgroup) don't. 

As an application of our results, the second author has been working on determining the $N$ for which $X_0(N)$ have infinitely many quartic points. The main difficulty in solving the problem is determining whether a given $X_0(N)$ has a degree $4$ map to an elliptic curve.

\section{Summary of results}
We summarize our results in the following tables. For each value of $N$, there are $7$ entries, listed in the order that they appear in: the genus $g$, the gonality of $X_0(N)$ over $\Q$ (denoted by $\gon_\Q$), references to how the lower and upper bound for the $\Q$-gonality were obtained (denoted by LB and UB, respectively), the gonality of $X_0(N)$ over $\C$ (denoted by $\gon_\C$) and finally references to how the lower and upper bound for the $\C$-gonality were obtained (again denoted by LB and UB, respectively).


For larger $N$, we show only those whose $\Q$-gonality we have determined and skip the others.

\clearpage
\begin{table}[ht]
\centering
\begin{tabular}{|c|c|c|c|c|c|c|c||c|c|c|c|c|c|c|c|}
  \hline
  $N$ & $g$ & ${\scriptstyle \gon_\Q}$  & LB & UB & ${\scriptstyle \gon_\C}$ & LB & UB & $N$ & $g$ & ${\scriptstyle \gon_\Q}$ & LB & UB & ${\scriptstyle \gon_\C}$ & LB & UB\\
  \hline

    $\leq10$ & $0$ & $1$ & & & $1$ & & &
    $11$ & $1$ & $2$ & \cite{Ogg74} & \cite{Ogg74} & $2$ & \cite{Ogg74} & \cite{Ogg74}\\
    $12$ & $0$ & $1$ & & & $1$ & & &
    $13$ & $0$ & $1$ & & & $1$ & &\\
    $14$ & $1$ & $2$ & \cite{Ogg74} & \cite{Ogg74} & $2$ & \cite{Ogg74} & \cite{Ogg74} &
    $15$ & $1$ & $2$ & \cite{Ogg74} & \cite{Ogg74} & $2$ & \cite{Ogg74} & \cite{Ogg74}\\
    $16$ & $0$ & $1$ & & & $1$ & & &
    $17$ & $1$ & $2$ & \cite{Ogg74} & \cite{Ogg74} & $2$ & \cite{Ogg74} & \cite{Ogg74}\\
    $18$ & $0$ & $1$ & & & $1$ & & &
    $19$ & $1$ & $2$ & \cite{Ogg74} & \cite{Ogg74} & $2$ & \cite{Ogg74} & \cite{Ogg74}\\
    $20$ & $1$ & $2$ & \cite{Ogg74} & \cite{Ogg74} & $2$ & \cite{Ogg74} & \cite{Ogg74} &
    $21$ & $1$ & $2$ & \cite{Ogg74} & \cite{Ogg74} & $2$ & \cite{Ogg74} & \cite{Ogg74}\\
    $22$ & $2$ & $2$ & \cite{Ogg74} & \cite{Ogg74} & $2$ & \cite{Ogg74} & \cite{Ogg74} &
    $23$ & $2$ & $2$ & \cite{Ogg74} & \cite{Ogg74} & $2$ & \cite{Ogg74} & \cite{Ogg74}\\
    $24$ & $1$ & $2$ & \cite{Ogg74} & \cite{Ogg74} & $2$ & \cite{Ogg74} & \cite{Ogg74} &
    $25$ & $0$ & $1$ & & & $1$ & &\\
    $26$ & $2$ & $2$ & \cite{Ogg74} & \cite{Ogg74} & $2$ & \cite{Ogg74} & \cite{Ogg74} &
    $27$ & $1$ & $2$ & \cite{Ogg74} & \cite{Ogg74} & $2$ & \cite{Ogg74} & \cite{Ogg74}\\
    $28$ & $2$ & $2$ & \cite{Ogg74} & \cite{Ogg74} & $2$ & \cite{Ogg74} & \cite{Ogg74} &
    $29$ & $2$ & $2$ & \cite{Ogg74} & \cite{Ogg74} & $2$ & \cite{Ogg74} & \cite{Ogg74}\\
    $30$ & $3$ & $2$ & \cite{Ogg74} & \cite{Ogg74} & $2$ & \cite{Ogg74} & \cite{Ogg74} &
    $31$ & $2$ & $2$ & \cite{Ogg74} & \cite{Ogg74} & $2$ & \cite{Ogg74} & \cite{Ogg74}\\
    $32$ & $1$ & $2$ & \cite{Ogg74} & \cite{Ogg74} & $2$ & \cite{Ogg74} & \cite{Ogg74} &
    $33$ & $3$ & $2$ & \cite{Ogg74} & \cite{Ogg74} & $2$ & \cite{Ogg74} & \cite{Ogg74}\\
    $34$ & $3$ & $3$ & ${\scriptstyle \gon_\C}$ & \cite{HasegawaShimura_trig} & $3$ & \cite{HasegawaShimura_trig} & \cite{HasegawaShimura_trig}&
    $35$ & $3$ & $2$ & \cite{Ogg74} & \cite{Ogg74} & $2$ & \cite{Ogg74} & \cite{Ogg74}\\
    $36$ & $1$ & $2$ & \cite{Ogg74} & \cite{Ogg74} & $2$ & \cite{Ogg74} & \cite{Ogg74} &
    $37$ & $2$ & $2$ & \cite{Ogg74} & \cite{Ogg74} & $2$ & \cite{Ogg74} & \cite{Ogg74}\\
    $38$ & $4$ & $4$ & \ref{prop:fp} & \cite{HasegawaShimura_trig} & $3$ & \cite{HasegawaShimura_trig} & \cite{HasegawaShimura_trig} &
    $39$ & $3$ & $2$ & \cite{Ogg74} & \cite{Ogg74} & $2$ & \cite{Ogg74} & \cite{Ogg74}\\
    $40$ & $3$ & $2$ & \cite{Ogg74} & \cite{Ogg74} & $2$ & \cite{Ogg74} & \cite{Ogg74} &
    $41$ & $3$ & $2$ & \cite{Ogg74} & \cite{Ogg74} & $2$ & \cite{Ogg74} & \cite{Ogg74}\\
    $42$ & $5$ & $4$ & ${\scriptstyle \gon_\C}$ & \ref{prop4.6} & $4$ &  \cite{JeonPark05} & \cite{JeonPark05}  &
    $43$ & $3$ & $3$ & ${\scriptstyle \gon_\C}$ & \cite{HasegawaShimura_trig} & $3$ & \cite{HasegawaShimura_trig} & \cite{HasegawaShimura_trig}\\
    $44$ & $4$ & $4$ & \ref{prop:fp} & \cite{HasegawaShimura_trig} & $3$ & \cite{HasegawaShimura_trig}& \cite{HasegawaShimura_trig} &
    $45$ & $3$ & $3$ & ${\scriptstyle \gon_\C}$ & \cite{HasegawaShimura_trig} & $3$ & \cite{HasegawaShimura_trig} & \cite{HasegawaShimura_trig}\\
    $46$ & $5$ & $2$ & \cite{Ogg74} & \cite{Ogg74} & $2$ & \cite{Ogg74} & \cite{Ogg74} &
    $47$ & $4$ & $2$ & \cite{Ogg74} & \cite{Ogg74} & $2$ & \cite{Ogg74} & \cite{Ogg74}\\
    $48$ & $3$ & $2$ & \cite{Ogg74} & \cite{Ogg74} & $2$ & \cite{Ogg74} & \cite{Ogg74} &
    $49$ & $1$ & $2$ & \cite{Ogg74} & \cite{Ogg74} & $2$ & \cite{Ogg74} & \cite{Ogg74}\\
    $50$ & $2$ & $2$ & \cite{Ogg74} & \cite{Ogg74} & $2$ & \cite{Ogg74} & \cite{Ogg74} &
    $51$ & $5$ & $4$ & ${\scriptstyle \gon_\C}$ & \ref{prop4.5} & $4$ & \cite{JeonPark05} & \cite{JeonPark05} \\
    $52$ & $5$ & $4$ & ${\scriptstyle \gon_\C}$ & \ref{prop4.6} & $4$ & \cite{JeonPark05} & \cite{JeonPark05} &
    $53$ & $4$ & $4$ & \ref{prop:fp} & \cite{HasegawaShimura_trig} & $3$ &\cite{HasegawaShimura_trig} & \cite{HasegawaShimura_trig}\\
    $54$ & $4$ & $3$ & ${\scriptstyle \gon_\C}$  & \cite{HasegawaShimura_trig} & $3$ & \cite{HasegawaShimura_trig} & \cite{HasegawaShimura_trig} & 
    $55$ & $5$ & $4$ & ${\scriptstyle \gon_\C}$ & \ref{prop4.5} & $4$ & \cite{JeonPark05} & \cite{JeonPark05}\\
    $56$ & $5$ & $4$ & ${\scriptstyle \gon_\C}$ & \ref{prop4.5} & $4$ & \cite{JeonPark05} & \cite{JeonPark05} & 
    $57$ & $5$ & $4$ & ${\scriptstyle \gon_\C}$ & \ref{prop4.6} & $4$ & \cite{JeonPark05} & \cite{JeonPark05}\\
    $58$ & $6$ & $4$ & ${\scriptstyle \gon_\C}$ & \ref{prop4.6} & $4$ & \cite{JeonPark05} & \cite{JeonPark05} & 
    $59$ & $5$ & $2$ & \cite{Ogg74} & \cite{Ogg74} & $2$ & \cite{Ogg74} & \cite{Ogg74}\\
    $60$ & $7$ & $4$ & ${\scriptstyle \gon_\C}$ & \ref{prop4.6} & $4$ & \cite{JeonPark05} & \cite{JeonPark05} & 
    $61$ & $4$ & $4$ & \ref{prop:fp} & \cite{HasegawaShimura_trig} & $3$ & \cite{HasegawaShimura_trig} &\cite{HasegawaShimura_trig}\\
    $62$ & $7$ & $4$ & ${\scriptstyle \gon_\C}$ & \ref{prop4.5} & $4$ & \cite{JeonPark05} & \cite{JeonPark05} & 
    $63$ & $5$ & $4$ & ${\scriptstyle \gon_\C}$ & \ref{prop4.5} & $4$ & \cite{JeonPark05} & \cite{JeonPark05}\\
    $64$ & $3$ & $3$ & ${\scriptstyle \gon_\C}$ & \cite{HasegawaShimura_trig} & $3$ & \cite{HasegawaShimura_trig} & \cite{HasegawaShimura_trig}& 
    $65$ & $5$ & $4$ & ${\scriptstyle \gon_\C}$ & \ref{prop4.5} & $4$ & \cite{JeonPark05} & \cite{JeonPark05}\\
    $66$ & $9$ & $4$ & ${\scriptstyle \gon_\C}$ & \ref{prop4.6} & $4$ & \cite{JeonPark05} & \cite{JeonPark05} & 
    $67$ & $5$ & $4$ & ${\scriptstyle \gon_\C}$ & \ref{prop4.6} & $4$ & \cite{JeonPark05} & \cite{JeonPark05}\\
    $68$ & $7$ & $4$ & ${\scriptstyle \gon_\C}$ & \ref{prop4.6} & $4$ & \cite{JeonPark05} & \cite{JeonPark05} & 
    $69$ & $7$ & $4$ & ${\scriptstyle \gon_\C}$ & \ref{prop4.5} & $4$ & \cite{JeonPark05} & \cite{JeonPark05}\\
    $70$ & $9$ & $4$ & ${\scriptstyle \gon_\C}$ & \ref{prop4.6} & $4$ & \cite{JeonPark05} & \cite{JeonPark05} & 
    $71$ & $6$ & $2$ & \cite{Ogg74} & \cite{Ogg74} & $2$ & \cite{Ogg74} & \cite{Ogg74}\\
    $72$ & $5$ & $4$ & ${\scriptstyle \gon_\C}$ & \ref{prop4.9} & $4$ & \cite{JeonPark05} & \cite{JeonPark05} & 
    $73$ & $5$ & $4$ & ${\scriptstyle \gon_\C}$ & \ref{prop4.6} & $4$ & \cite{JeonPark05} & \cite{JeonPark05}\\
    $74$ & $8$ & $4$ & ${\scriptstyle \gon_\C}$ & \ref{prop4.6} & $4$ & \cite{JeonPark05} & \cite{JeonPark05} & 
    $75$ & $5$ & $4$ & ${\scriptstyle \gon_\C}$ & \ref{prop4.5} & $4$ & \cite{JeonPark05} & \cite{JeonPark05}\\
    $76$ & $8$ & $6$ & \ref{prop:fp} & \ref{prop4.10} & $5$ & \cite{JeonPark05} & \ref{prop_A.1} &
    $77$ & $7$ & $4$ & ${\scriptstyle \gon_\C}$ & \ref{prop4.6} & $4$ & \cite{JeonPark05} & \cite{JeonPark05}\\
    $78$ & $11$ & $4$ & ${\scriptstyle \gon_\C}$ & \cite{HasegawaShimura_trig} & $4$ & \cite{JeonPark05} & \cite{JeonPark05} & 
    $79$ & $6$ & $4$ & ${\scriptstyle \gon_\C}$ & \ref{prop4.5} & $4$ & \cite{JeonPark05} & \cite{JeonPark05}\\
    $80$ & $7$ & $4$ & ${\scriptstyle \gon_\C}$ & \ref{prop4.6} & $4$ & \cite{JeonPark05} & \cite{JeonPark05} & 
    $81$ & $4$ & $3$ & ${\scriptstyle \gon_\C}$ & \cite{HasegawaShimura_trig} & $3$ & \cite{HasegawaShimura_trig} & \cite{HasegawaShimura_trig}\\
    $82$ & $9$ & $6$ & \ref{prop:fp} & \ref{prop4.9} & $[5,6]$ & \cite{JeonPark05} & ${\scriptstyle \gon_\Q}$ & 
    $83$ & $5$ & $4$ & ${\scriptstyle \gon_\C}$ & \ref{prop4.5} & $4$ & \cite{JeonPark05} & \cite{JeonPark05}\\
    $84$ & $11$ & $6$ & \ref{prop:fp} & \ref{cor4.16_UB} & $6$ & \ref{Cgon6} & ${\scriptstyle \gon_\Q}$ &
    $85$ & $7$ & $4$ & ${\scriptstyle \gon_\C}$ & \ref{prop4.7} & $4$ & \cite{JeonPark05} & \cite{JeonPark05}\\
    $86$ & $10$ & $6$ & \ref{prop:fp} & \ref{prop4.10} & $6$ & \ref{Cgon6} & ${\scriptstyle \gon_\Q}$ & 
    $87$ & $9$ & $4$ & ${\scriptstyle \gon_\C}$ & \ref{prop4.6} & $4$ & \cite{JeonPark05} & \cite{JeonPark05} \\
    $88$ & $9$ & $4$ & ${\scriptstyle \gon_\C}$ & \ref{prop4.7} & $4$ & \cite{JeonPark05} & \cite{JeonPark05} &
    $89$ & $7$ & $4$ & ${\scriptstyle \gon_\C}$ & \ref{prop4.5} & $4$ & \cite{JeonPark05} & \cite{JeonPark05} \\ 
    $90$ & $11$ & $6$ & ${\scriptstyle \gon_\C}$ & \ref{prop4.9} & $6$ & \ref{prop4.24} & ${\scriptstyle \gon_\Q}$ &
    $91$ & $7$ & $4$ & ${\scriptstyle \gon_\C}$ & \ref{prop4.6} & $4$ & \cite{JeonPark05} & \cite{JeonPark05}  \\
    $92$ & $10$ & $4$ & ${\scriptstyle \gon_\C}$ & \ref{prop4.5} & $4$ & \cite{JeonPark05} & \cite{JeonPark05} &
    $93$ & $9$ & $6$ & \ref{prop:fp} & \ref{cor4.16_UB} & $6$ & \ref{Cgon6} &  ${\scriptstyle \gon_\Q}$ \\

    \hline
\end{tabular}\\
\vspace{5mm}
\caption{Gonalities of $X_0(N)$, for $N\leq93$.}
\label{tab:main}
\end{table}

\clearpage
\begin{table}[ht]
\centering
\begin{tabular}{|c|c|c|c|c|c|c|c||c|c|c|c|c|c|c|c|}
  \hline
  $N$ & $g$ & ${\scriptstyle \gon_\Q}$  & LB & UB & ${\scriptstyle \gon_\C}$ & LB & UB & $N$ & $g$ & ${\scriptstyle \gon_\Q}$ & LB & UB & ${\scriptstyle \gon_\C}$ & LB & UB\\
  \hline
  
    $94$ & $11$ & $4$ & ${\scriptstyle \gon_\C}$ & \cite{HasegawaShimura_trig} & $4$ & \cite{JeonPark05} & \cite{JeonPark05} &
    $95$ & $9$ & $4$ & ${\scriptstyle \gon_\C}$ & \ref{prop4.5} & $4$ & \cite{JeonPark05} & \cite{JeonPark05} \\
    $96$ & $9$ & $4$ & ${\scriptstyle \gon_\C}$ & \ref{prop4.9} & $4$ & \cite{JeonPark05} & \cite{JeonPark05} &
    $97$ & $7$ & $6$ & \ref{prop:97_133} & \ref{prop4.10} & $5$ & \cite{JeonPark05} & \ref{prop_A.1} \\
    $98$ & $7$ & $4$ & ${\scriptstyle \gon_\C}$ & \ref{prop4.6} & $4$ & \cite{JeonPark05} & \cite{JeonPark05} &
    $99$ & $9$ & $6$ & \ref{prop4.26} &  \ref{prop4.9} & $4$ & \cite{JeonPark05} & \cite{JeonPark05} \\
    $100$ & $7$ & $4$ & ${\scriptstyle \gon_\C}$ & \ref{prop4.6} & $4$ & \cite{JeonPark05} & \cite{JeonPark05} &
    $101$ & $8$ & $4$ & ${\scriptstyle \gon_\C}$ & \ref{prop4.5} & $4$ & \cite{JeonPark05} & \cite{JeonPark05}\\
    $102$ & $15$ & $8$ & \ref{prop:fp} & \ref{prop4.11} & $[6,8]$ & \ref{Cgon6cs} & ${\scriptstyle \gon_\Q}$ &
    $103$ & $8$ & $4$ & ${\scriptstyle \gon_\C}$ & \ref{prop4.6} & $4$ & \cite{JeonPark05} & \cite{JeonPark05}\\
    $104$ & $11$ & $4$ & ${\scriptstyle \gon_\C}$ & \cite{HasegawaShimura_trig} & $4$ & \cite{JeonPark05} & \cite{JeonPark05} &
    $105$ & $13$ & $6$ & \ref{prop4.13} & \ref{prop4.13_UB} & $6$ & \ref{prop4.13} & ${\scriptstyle \gon_\Q}$\\
    $106$ & $12$ & $8$ & \ref{prop:fp}& \ref{prop4.11}& $[6,7]$ &\ref{Cgon6} &\ref{prop_A.1} &
    $107$ & $9$ & $4$ & ${\scriptstyle \gon_\C}$ & \ref{prop4.6} & $4$ & \cite{JeonPark05} & \cite{JeonPark05}\\
    $108$ & $10$ & $6$ & \ref{prop:fp} & \ref{prop4.5} & $[5,6]$ & \cite{JeonPark05} & ${\scriptstyle \gon_\Q}$ &
    $109$ & $8$ & $5$ & \ref{prop:fp} & \ref{prop4.14} & $4$ & \cite{JeonPark05} & \cite{JeonPark05} \\
    $110$ & $15$ & $8$ & \ref{cor4.18} & \ref{prop4.10} & $6$ & \ref{cor4.18} & \ref{prop4.10} &
    $111$ & $9$ & $4$ & ${\scriptstyle \gon_\C}$ & \cite{HasegawaShimura_trig} & $4$ & \cite{JeonPark05} & \cite{JeonPark05}\\
    $112$ & $11$ & $6$ & \ref{prop:fp} & \ref{prop4.28} & $[5,6]$ & \cite{JeonPark05} & ${\scriptstyle \gon_\Q}$ &
    $113$ & $9$ & $6$ & \ref{prop:fp} & \ref{prop4.10} & $[5,6]$ & \cite{JeonPark05} & ${\scriptstyle \gon_\Q}$\\
    $114$ & $17$ & $8$ & \ref{prop:fp} & \ref{prop4.11} & $[6,8]$ & \ref{prop:gon_ge_6} &  ${\scriptstyle \gon_\Q}$ &
    $115$ & $11$ & $6$ & \ref{prop:fp} & \ref{cor4.16_UB} & $6$ & \ref{Cgon6} & ${\scriptstyle \gon_\Q}$\\
    $116$ & $13$ & $6$ & \ref{prop4.13} & \ref{cor4.16_UB} & $6$ & \ref{prop4.13} & ${\scriptstyle \gon_\Q}$ &
    $117$ & $11$ & $6$ & \ref{prop:fp} & \ref{prop4.9} & $[5,6]$ & \cite{JeonPark05} & ${\scriptstyle \gon_\Q}$\\
    $118$ & $14$ & $6$ & \ref{prop4.13} & \ref{prop4.13_UB} & $6$ & \ref{prop4.13} & ${\scriptstyle \gon_\Q}$ &
    $119$ & $11$ & $4$ & ${\scriptstyle \gon_\C}$ & \cite{HasegawaShimura_trig} & $4$ & \cite{JeonPark05} & \cite{JeonPark05}\\
    $120$ & $17$ & $8$ & \ref{prop4.20} & \ref{prop4.11} & $8$ & \ref{prop4.20} & ${\scriptstyle \gon_\Q}$ &
    $121$ & $6$ & $4$ & ${\scriptstyle \gon_\C}$ & \ref{prop4.6} & $4$ & \cite{JeonPark05} & \cite{JeonPark05}\\
    $122$ & $14$ & $6$ & \ref{prop:fp} & \ref{prop4.13_UB} & $[5,6]$ & \cite{JeonPark05} & ${\scriptstyle \gon_\Q}$ &
    $123$ & $13$ & $6$ & \ref{prop4.13} & \ref{prop4.13_UB} & $6$ & \ref{prop4.13} &${\scriptstyle \gon_\Q}$ \\
    $124$ & $14$ & $6$ & \ref{prop4.13} & \ref{prop4.13_UB} & $6$ & \ref{prop4.13} & ${\scriptstyle \gon_\Q}$ &
    $125$ & $8$ & $4$ & ${\scriptstyle \gon_\C}$ & \ref{prop4.6} & $4$ & \cite{JeonPark05} & \cite{JeonPark05}\\
    $126$ & $17$ & $8$ & \ref{prop4.20} & \ref{prop4.11} & $8$ & \ref{prop4.20} & ${\scriptstyle \gon_\Q}$ &
    $127$ & $10$ & $6$ & \ref{prop:fp} & \ref{prop4.10} & $6$ & \ref{Cgon6} & ${\scriptstyle \gon_\Q}$\\
    $128$ & $9$ & $6$ & \ref{prop:fp} & \ref{prop4.10} & $6$ & \ref{Cgon6} & ${\scriptstyle \gon_\Q}$ &
    $129$ & $13$ & $6$ & \ref{prop:fp} & \ref{cor4.16_UB} & $6$ &\ref{Cgon6cs} &${\scriptstyle \gon_\Q}$ \\
    $130$ & $17$ & $8$ & \ref{prop:130} & \ref{prop4.11} & $[6,8]$ & \ref{prop:gon_ge_6} & ${\scriptstyle \gon_\Q}$ & 
    $131$ & $11$ & $4$ & ${\scriptstyle \gon_\C}$ & \cite{HasegawaShimura_trig} & $4$ & \cite{JeonPark05} & \cite{JeonPark05}\\
    $132$ & $19$ & $8$ & \ref{prop:fp} & \ref{prop4.9}& $[6,8]$ & \ref{prop:gon_ge_6} & ${\scriptstyle \gon_\Q}$ &
    $133$ & $11$ & $8$ &\ref{prop:97_133} & \ref{prop4.10}& $6$ & \ref{Cgon6} & \ref{prop4.10} \\
    $134$ & $16$ & $8$ &\ref{prop:fp} & \ref{prop4.10} & $[6,8]$ & \ref{prop:gon_ge_6} & ${\scriptstyle \gon_\Q}$ &
    $135$ & $13$ & $6$ & ${\scriptstyle \gon_\C}$ &\ref{prop4.10} & $6$ & $\ref{prop:gon_ge_6}$ & ${\scriptstyle \gon_\Q}$ \\
    $136$ & $15$ & $8$ &\ref{prop:fp} & \ref{prop4.9} & $[5,8]$ & \cite{JeonPark05} & ${\scriptstyle \gon_\Q}$ &
    $137$ & $11$ & $6$ &\ref{prop:fp} & \ref{cor4.16_UB}& $6$ & \ref{Cgon6} & ${\scriptstyle \gon_\Q}$ \\
    $138$ & $21$ & $8$ &\ref{prop4.20} &\ref{prop4.11} & $8$ & \ref{prop4.20} & ${\scriptstyle \gon_\C}$&
    $139$ & $11$ & $6$ &\ref{prop4.13} & \ref{prop4.10}& $6$ & \ref{prop4.13} & ${\scriptstyle \gon_\Q}$ \\
    $140$ & $19$ & $8$ &\ref{prop:fp} &\ref{prop4.9} & $[6,8]$ & \ref{prop:gon_ge_6} & ${\scriptstyle \gon_\Q}$ &
    $141$ & $15$ & $6$ &\ref{prop4.13} & \ref{prop4.10} & $6$ & \ref{prop4.13} & ${\scriptstyle \gon_\Q}$ \\
    $142$ & $17$ & $4$ & ${\scriptstyle \gon_\C}$ & \cite{HasegawaShimura_trig} &$4$ & \cite{JeonPark05} & \cite{JeonPark05}   &
    $143$ & $13$ & $4$ & ${\scriptstyle \gon_\C}$ & \cite{HasegawaShimura_trig} & $4$ &\cite{JeonPark05} & \cite{JeonPark05}   \\
    $144$ & $13$ & $6$ &\ref{prop:fp} & \ref{prop4.9}& $[5,6]$ & \cite{JeonPark05} & ${\scriptstyle \gon_\Q}$ &
    $145$ & $13 $ & $[7,8]$ & \ref{prop:97_133} & \ref{prop4.10} & $6$ &\ref{prop:gon_ge_6} & \ref{prop4.10} \\
    $146$ & $17$ & $6$ &\ref{prop4.13} & \ref{prop4.13_UB}& $6$ & \ref{prop4.13} & ${\scriptstyle \gon_\Q}$ &
    $147$ & $11$ & $6$ &\ref{prop:fp} & \ref{prop4.13_UB}& $[5,6]$ & \cite{JeonPark05} & ${\scriptstyle \gon_\Q}$  \\
    $148$ & $17$ & $8$ &\ref{prop:fp} & \ref{prop4.9}& $[5,8]$ & \cite{JeonPark05} & ${\scriptstyle \gon_\Q}$ &
    $149$ & $12$ & $6$ &\ref{prop4.13} & \ref{prop4.10}& $6$ & \ref{prop4.13} & ${\scriptstyle \gon_\Q}$ \\
    $150$ & $19$ & $8$ &\ref{prop:fp} & \ref{prop4.9} & $[6,8]$ & \ref{Cgon6cs} & ${\scriptstyle \gon_\Q}$ & 
    $151$ & $12$ & $6$ & \ref{prop:fp} & \ref{prop4.10} & $6$ & \ref{prop:gon_ge_6} & ${\scriptstyle \gon_\Q}$ \\
    $152$ & $17$ & $8$ & \ref{prop:fp} & \ref{tetragonal_quotients} & $[6,8]$ & \ref{Cgon6cs} & ${\scriptstyle \gon_\Q}$ &
    $153$ & $15$ & $8$ & \ref{prop:fp} & \ref{prop4.11} & $[5,8]$ & \cite{JeonPark05} & ${\scriptstyle \gon_\Q}$ \\
    $154$ & $21$ & $[8,12]$ & \ref{prop:fp} & \ref{prop4.9} & $[5,12]$ & \cite{JeonPark05} & ${\scriptstyle \gon_\Q}$ &
    $155$ & $15$ & $6$ & ${\scriptstyle \gon_\C}$ & \ref{cor4.16_UB} & $6$ & \ref{Cgon6cs} & ${\scriptstyle \gon_\Q}$\\
    $156$ & $23$ & $8$ & \ref{prop4.20} & \ref{prop4.9} & $8$ & \ref{prop4.20} & ${\scriptstyle \gon_\Q}$ &
    $157$ & $12$ & $8$ & \ref{prop:fp} & \ref{tetragonal_quotients} & $[5,7]$ & \cite{JeonPark05} & \ref{prop_A.1}\\
    $158$ & $19$ & $8$ & \ref{prop4.20} & \ref{prop4.11} & $8$ & \ref{prop4.20} & ${\scriptstyle \gon_\Q}$ &
    $159$ & $17$ & $6$ &  ${\scriptstyle \gon_\C}$ & \ref{cor4.16_UB} & $6$ & \ref{Cgon6cs} & ${\scriptstyle \gon_\Q}$\\
    $160$ & $17$ & $8$ & \ref{prop:fp} & \ref{prop4.9} & $[6,8]$ & \ref{prop:gon_ge_6} & ${\scriptstyle \gon_\Q}$ &
    $161$ & $15$ & $8$ & \ref{cor4.18} & \ref{prop4.10} & $6$ & \ref{cor4.18} & \ref{prop4.10} \\
    $162$ & $16$ & $6$ & \ref{prop:fp} & \ref{prop4.13_UB} & $[5,6]$ & \cite{JeonPark05} & ${\scriptstyle \gon_\Q}$ &
    $163$ & $13$ & $[7,8]$ & \ref{prop:fp} & \ref{tetragonal_quotients} & $[5,8]$ & \cite{JeonPark05} & ${\scriptstyle \gon_\Q}$\\
    $164$ & $19$ & $6$ & \ref{prop4.13} & \ref{prop4.13_UB} & $6$ & \ref{prop4.13} & ${\scriptstyle \gon_\Q}$ &
    $165$ & $21$ & $8$ & \ref{prop4.20} & \ref{prop4.11} & $8$ & \ref{prop4.20} & ${\scriptstyle \gon_\Q}$\\
    $166$ & $20$ & $8$ & \ref{prop4.20} & \ref{prop4.11} & $8$ & \ref{prop4.20} & ${\scriptstyle \gon_\Q}$&
    $167$ & $14$ & $4$ & \cite{HasegawaShimura_trig} & \cite{HasegawaShimura_trig} & $4$ & \cite{JeonPark05} & \cite{JeonPark05}\\
    $168$ & $25$ & $8$ & \ref{prop4.20} & \ref{prop4.11} & $8$ & \ref{prop4.20} & ${\scriptstyle \gon_\Q}$ &   
    $169$ & $8$ & $6$ & \ref{prop:fp} & \ref{prop4.10} & $5$ & \cite{JeonPark05} & \ref{prop_A.1}\\
 
    \hline
\end{tabular}\\
\vspace{5mm}
\caption{Gonalities of $X_0(N)$, for $94\leq N\leq 169$.}
\label{tab:main2}
\end{table}

\clearpage
\begin{table}[ht]
\centering
\begin{tabular}{|c|c|c|c|c|c|c|c||c|c|c|c|c|c|c|c|}
  \hline
  $N$ & $g$ & ${\scriptstyle \gon_\Q}$  & LB & UB & ${\scriptstyle \gon_\C}$ & LB & UB & $N$ & $g$ & ${\scriptstyle \gon_\Q}$ & LB & UB & ${\scriptstyle \gon_\C}$ & LB & UB\\
  \hline

    $171$ & $17$ & $8$ & \ref{prop4.20} & \ref{prop4.11} & $8$ & \ref{prop4.20} & ${\scriptstyle \gon_\Q}$ &
    $173$ & $14$ & $8$ & \ref{cor4.18} & \ref{prop4.10} & $6$ & \ref{cor4.18} & \ref{prop4.10}\\
    $175$ & $15$ & $8$ & \ref{prop:fp} & \ref{prop4.9} & $[6,8]$ & \ref{Cgon6cs} & ${\scriptstyle \gon_\Q}$ & 
    $176$ & $19$ & $8$ & \ref{prop:fp} & \ref{prop4.9} & $[6,8]$ & \ref{prop:gon_ge_6} & ${\scriptstyle \gon_\Q}$\\
    $177$ & $19$ & $8$ & \ref{cor4.18} & \ref{prop4.10} & $6$ & \ref{cor4.18} & \ref{prop4.10} &
    $179$ & $15$ & $6$ & \ref{prop:fp} & \ref{prop4.10} & $[5,6]$ & \cite{JeonPark05} & ${\scriptstyle \gon_\Q}$\\
    $181$ & $14$ & $6$ & \ref{prop:fp} & \ref{prop4.13_UB} & $[5,6]$ & \cite{JeonPark05} & ${\scriptstyle \gon_\Q}$ &
    $183$ & $19$ & $8$ & \ref{prop4.20} & \ref{tetragonal_quotients} & $8$ & \ref{prop4.20} & ${\scriptstyle \gon_\Q}$ \\
    $184$ & $21$ & $8$ & \ref{prop4.20} & \ref{prop4.9} & $8$ & \ref{prop4.20} & ${\scriptstyle \gon_\Q}$ &
    $185$ & $17$ & $8$ & \ref{prop4.20} & \ref{tetragonal_quotients} & $8$ & \ref{prop4.20} & ${\scriptstyle \gon_\Q}$ \\
    $188$ & $22$ & $8$ & \ref{cor4.18} & \ref{prop4.10} & $6$ & \ref{cor4.18} & \ref{prop4.10} &
    $190$ & $27$ & $8$ & \ref{prop4.20} & \ref{prop4.11} & $8$ & \ref{prop4.20} & ${\scriptstyle \gon_\Q}$ \\
    $191$ & $16$ & $4$ & \cite{HasegawaShimura_trig} & \cite{HasegawaShimura_trig} & $4$ & \cite{JeonPark05} & \cite{JeonPark05} &
    $192$ & $21$ & $8$ & \ref{prop:fp} & \ref{prop4.9} & $[6,8]$ & \cite{HasegawaShimura_trig} & ${\scriptstyle \gon_\Q}$ \\
    $195$ & $25$ & $8$ & \ref{prop4.20} & \ref{prop4.11} & $8$ & \ref{prop4.20} & ${\scriptstyle \gon_\Q}$ &
    $196$ & $17$ & $8$ & \ref{prop:fp} & \ref{prop4.9} & $[6,8]$ & \ref{prop:gon_ge_6} & ${\scriptstyle \gon_\Q}$ \\
    $197$ & $16$ & $8$ & \ref{prop4.20} & \ref{tetragonal_quotients} & $8$ & \ref{prop4.20} & ${\scriptstyle \gon_\Q}$ &
    $199$ & $16$ & $8$ & \ref{cor4.18} & \ref{cor4.18} & $6$ & \ref{cor4.18} & ${\scriptstyle \gon_\Q}$ \\
    $200$ & $19$ & $8$ & \ref{prop:fp} & \ref{prop4.9} & $[6,8]$ & \ref{prop:gon_ge_6} & ${\scriptstyle \gon_\Q}$ &
    $203$ & $19$ & $8$ & \ref{prop4.20} & \ref{tetragonal_quotients} & $8$ & \ref{prop4.20} & ${\scriptstyle \gon_\Q}$ \\
    $205$ & $19$ & $8$ & \ref{prop4.20} & \ref{prop4.11} & $8$ & \ref{prop4.20} & ${\scriptstyle \gon_\Q}$ &
    $206$ & $25$ & $8$ & \ref{prop4.20} & \ref{prop4.11} & $8$ & \ref{prop4.20} & ${\scriptstyle \gon_\Q}$ \\
    $209$ & $19$ & $8$ & \ref{prop4.20} & \ref{prop4.11} & $8$ & \ref{prop4.20} & ${\scriptstyle \gon_\Q}$ &
    $211$ & $17$ & $8$ & \ref{prop:fp} & \ref{tetragonal_quotients} & $[6,8]$ & \ref{Cgon6cs} & ${\scriptstyle \gon_\Q}$ \\
    $213$ & $23$ & $8$ & \ref{prop4.20} & \ref{prop4.11} & $8$ & \ref{prop4.20} & ${\scriptstyle \gon_\Q}$ &
    $215$ & $21$ & $6$ & ${\scriptstyle \gon_\C}$ & \ref{prop4.10} & $6$ & \ref{prop:gon_ge_6} & ${\scriptstyle \gon_\Q}$ \\
    $221$ & $19$ & $8$ & \ref{prop4.20} & \ref{prop4.11} & $8$ & \ref{prop4.20} & ${\scriptstyle \gon_\Q}$ &
    $223$ & $18$ & $8$ & \ref{prop:fp} & \ref{tetragonal_quotients} & $[7,8]$ & \ref{Cgon6cs} & ${\scriptstyle \gon_\Q}$ \\
    $227$ & $19$ & $6$ & \ref{prop4.13} & \ref{prop4.13_UB} & $6$ & \ref{prop4.13} & ${\scriptstyle \gon_\Q}$ &
    $239$ & $20$ & $6$ & \ref{prop4.13} & \ref{prop4.13_UB} & $6$ & \ref{prop4.13} & ${\scriptstyle \gon_\Q}$ \\
    $251$ & $21$ & $8$ & \ref{cor4.18} & \ref{prop4.10} & $6$ & \ref{cor4.18} & \ref{prop4.10} &
    $263$ & $22$ & $8$ & \ref{prop4.20} & \ref{tetragonal_quotients} & $8$ & \ref{prop4.20} & ${\scriptstyle \gon_\Q}$ \\
    $269$ & $22$ & $8$ & \ref{prop4.20} & \ref{tetragonal_quotients} & $8$ & \ref{prop4.20} & ${\scriptstyle \gon_\Q}$ &
    $271$ & $22$ & $10$ & \ref{Qgon10} & \ref{Qgon10} & $8$ & \ref{Qgon10} & \ref{Qgon10} \\
    $279$ & $29$ & $8$ & \ref{prop4.20} & \ref{prop4.11} & $8$ & \ref{prop4.20} & ${\scriptstyle \gon_\Q}$ &
    $284$ & $34$ & $8$ & \ref{prop4.20} & \ref{prop4.11} & $8$ & \ref{prop4.20} & ${\scriptstyle \gon_\Q}$ \\
    $287$ & $27$ & $8$ & \ref{prop4.20} & \ref{prop4.11} & $8$ & \ref{prop4.20} & ${\scriptstyle \gon_\Q}$ &
    $299$ & $27$ & $8$ & \ref{prop4.20} & \ref{prop4.11} & $8$ & \ref{prop4.20} & ${\scriptstyle \gon_\Q}$ \\
    $311$ & $26$ & $8$ & \ref{cor4.18} & \ref{prop4.10} & $6$ & \ref{cor4.18} & \ref{prop4.10} &
    $359$ & $30$ & $8$ & \ref{prop4.20} & \ref{tetragonal_quotients} & $8$ & \ref{prop4.20} & ${\scriptstyle \gon_\Q}$\\
    \hline
\end{tabular}\\

\vspace{5mm}
\caption{Gonalities of some $X_0(N)$, for $N>170$.}
\label{tab:large}
\end{table}

\bibliographystyle{siam}
\bibliography{bibliography1}

\def\cprime{$'$} \def\cprime{$'$}
\begin{thebibliography}{10}

\bibitem{abramovich}
{\sc D.~Abramovich}, {\em A linear lower bound on the gonality of modular
  curves}, Internat. Math. Res. Notices,  (1996), pp.~1005--1011.

\bibitem{quad_pts}
{\sc N.~Adžaga, T.~Keller, P.~Michaud-Jacobs, F.~Najman, E.~Ozman, and
  B.~Vukorepa}, {\em Computing quadratic points on modular curves {$X_0(N)$}}.
\newblock preprint, available at: \url{https://arxiv.org/abs/2303.12566}.

\bibitem{Bars99}
{\sc F.~Bars}, {\em Bielliptic modular curves}, J. Number Theory, 76 (1999),
  pp.~154--165.

\bibitem{BarsDalal22}
{\sc F.~Bars and T.~Dalal}, {\em Infinitely many cubic points for {$X_0^+(N)$}
  over {$\Q$}}.
\newblock preprint, available at: \url{https://arxiv.org/abs/2207.03727}.

\bibitem{magma}
{\sc W.~Bosma, J.~Cannon, and C.~Playoust}, {\em The {Magma} algebra system.
  {I}: {The} user language}, J. Symb. Comput., 24 (1997), pp.~235--265.

\bibitem{CoppensMartens91}
{\sc M.~Coppens and G.~Martens}, {\em Secant spaces and {Clifford}'s theorem},
  Compos. Math., 78 (1991), pp.~193--212.

\bibitem{Deg3Class}
{\sc M.~{Derickx}, A.~{Etropolski}, M.~{van Hoeij}, J.~S. {Morrow}, and
  D.~{Zureick-Brown}}, {\em {Sporadic cubic torsion}}, {Algebra Number Theory},
  15 (2021), pp.~1837--1864.

\bibitem{DerickxSutherland17}
{\sc M.~Derickx and A.~V. Sutherland}, {\em Torsion subgroups of elliptic
  curves over quintic and sextic number fields}, Proc. Am. Math. Soc., 145
  (2017), pp.~4233--4245.

\bibitem{derickxVH}
{\sc M.~Derickx and M.~van Hoeij}, {\em Gonality of the modular curve
  {$X_1(N)$}}, J. Algebra, 417 (2014), pp.~52--71.

\bibitem{FurumotoHasegawa1999}
{\sc M.~Furumoto and Y.~Hasegawa}, {\em {Hyperelliptic Quotients of Modular
  Curves $X_0(N)$}}, Tokyo Journal of Mathematics, 22 (1999), pp.~105 -- 125.

\bibitem{Green84}
{\sc M.~L. Green}, {\em Koszul cohomology and the geometry of projective
  varieties. {Appendix}: {The} nonvanishing of certain {Koszul} cohomology
  groups (by {Mark} {Green} and {Robert} {Lazarsfeld})}, J. Differ. Geom., 19
  (1984), pp.~125--167, 168--171.

\bibitem{HasegawaShimura_trig}
{\sc Y.~Hasegawa and M.~Shimura}, {\em Trigonal modular curves}, Acta Arith.,
  88 (1999), pp.~129--140.

\bibitem{HasegawaShimura1999}
\leavevmode\vrule height 2pt depth -1.6pt width 23pt, {\em Trigonal modular
  curves {{\(X_0^{+d}(N)\)}}}, Proc. Japan Acad., Ser. A, 75 (1999),
  pp.~172--175.

\bibitem{Hess02}
{\sc F.~Hess}, {\em Computing {Riemann}-{Roch} spaces in algebraic function
  fields and related topics.}, J. Symb. Comput., 33 (2002), pp.~425--445.

\bibitem{Jeon2021}
{\sc D.~Jeon}, {\em Modular curves with infinitely many cubic points}, J.
  Number Theory, 219 (2021), pp.~344--355.

\bibitem{JeonKimPark06}
{\sc D.~Jeon, C.~H. Kim, and E.~Park}, {\em On the torsion of elliptic curves
  over quartic number fields}, J. London Math. Soc. (2), 74 (2006), pp.~1--12.

\bibitem{JeonKimSchweizer04}
{\sc D.~Jeon, C.~H. Kim, and A.~Schweizer}, {\em On the torsion of elliptic
  curves over cubic number fields}, Acta Arith., 113 (2004), pp.~291--301.

\bibitem{JeonPark05}
{\sc D.~Jeon and E.~Park}, {\em Tetragonal modular curves}, Acta Arith., 120
  (2005), pp.~307--312.

\bibitem{katz81}
{\sc N.~M. Katz}, {\em Galois properties of torsion points on abelian
  varieties}, Invent. Math., 62 (1981), pp.~481--502.

\bibitem{lmfdb}
{\sc {LMFDB Collaboration}}, {\em The {L}-functions and modular forms
  database}.
\newblock \url{http://www.lmfdb.org}, 2022.
\newblock [Online; accessed 30 June 2022].

\bibitem{mazur77}
{\sc B.~Mazur}, {\em Modular curves and the {E}isenstein ideal}, Inst. Hautes
  \'Etudes Sci. Publ. Math.,  (1977), pp.~33--186 (1978).

\bibitem{Mestre81}
{\sc J.-F. Mestre}, {\em Corps euclidiens, unites exceptionnelles et courbes
  elliptiques}, J. Number Theory, 13 (1981), pp.~123--137.

\bibitem{NguyenSaito}
{\sc K.~V. Nguyen and M.-H. Saito}, {\em $d$-gonality of modular curves and
  bounding torsions}.
\newblock preprint, arXiv:alg-geom/9603024.

\bibitem{Ogg74}
{\sc A.~P. Ogg}, {\em Hyperelliptic modular curves}, Bull. Soc. Math. France,
  102 (1974), pp.~449--462.

\bibitem{Poonen2007}
{\sc B.~Poonen}, {\em Gonality of modular curves in characteristic {{\(p\)}}},
  Math. Res. Lett., 14 (2007), pp.~691--701.

\bibitem{Schreyer91}
{\sc F.-O. Schreyer}, {\em A standard basis approach to syzygies of canonical
  curves}, J. Reine Angew. Math., 421 (1991), pp.~83--123.

\bibitem{Stichtenoth09}
{\sc H.~Stichtenoth}, {\em Algebraic function fields and codes}, vol.~254 of
  Grad. Texts Math., Berlin: Springer, 2nd ed.~ed., 2009.

\end{thebibliography}

\end{document}